%% file: main.tex
\documentclass{article}

\usepackage{amsfonts, amsmath, amssymb, amsthm, array, booktabs, cancel, cite, color, enumerate, graphbox, graphics, graphicx, hyperref, import, mathrsfs, mathtools, pdfpages, url, tabularx, tikz-cd, tikz,tikz-3dplot,,url, varwidth, wrapfig}

\usepackage[shortlabels]{enumitem}

\usepackage[all]{xy}

\usepackage[nottoc]{tocbibind}

\usepackage[margin=10pt,font=small,labelsep=period]{caption}

\newlength{\notewidth}

\newcommand{\s}{\small}

\newtheorem{lemma}{Lemma}[section]
\newtheorem{prop}[lemma]{Proposition}

\newtheorem{thm}{Theorem}
\newtheorem{cor}[lemma]{Corollary}
\newtheorem{conjecture}[lemma]{Conjecture}

\theoremstyle{definition}
\newtheorem{rmrk}[lemma]{Remark}

\newcommand{\T}{{\mathbf T}}
\newcommand{\N}{{\mathbf N}}
\newcommand{\B}{{\mathbf B}}

\newcommand{\R}{{\mathbb R}}

\newcommand{\x}{{\bf x}}

\newcommand{\bfe}{{\bf n}}

\newcommand{\y}{{\bf y}}
\renewcommand{\v}{{\bf v}}
\newcommand{\p}{{\bf p}}
\renewcommand{\r}{{\bf r}}
\newcommand{\K}{{\bf K}}
\newcommand{\Iso}{\mathrm{Iso}}
\newcommand{\mn}{\medskip\noindent}
\newcommand{\be}{\begin{equation}}
\newcommand{\ee}{\end{equation}}

\newcommand{\bp}{\begin{proof}}
\newcommand{\ep}{\end{proof}}

\newcommand{\Kh}{Kirchhoff\ }
\newcommand{\Kr}{Kirchhoff rods}
\newcommand{\D}{{\mathscr D}}
\newcommand{\sR}{sub-Riemannian\ }
\renewcommand{\u}{{\mathbf u}}

\newcommand{\so}{\Rightarrow}

\title{Bicycling geodesics are \Kh rods}

\author{Gil Bor\footnote{
CIMAT, A.P. 402, Guanajuato, Gto. 36000, Mexico; 
gil@cimat.mx
}
\and
Connor Jackman\footnote{
CIMAT, A.P. 402, Guanajuato, Gto. 36000, Mexico; 
connor.jackman@cimat.mx
}
\and
Serge Tabachnikov\footnote{
Department of Mathematics,
Penn State University, 
University Park, PA 16802;
tabachni@math.psu.edu}
}

\date{\today}

\begin{document}

\maketitle

\begin{abstract}
A bicycle path is a pair of trajectories  in $\R^n$, the `front' and `back' tracks,   traced out by the endpoints of a moving line segment of fixed length (the `bicycle frame') and  tangent to the back track. Bicycle geodesics are bicycle  paths whose front track's length is critical among all bicycle paths connecting two given   placements of the line segment. 

We write down and study the associated variational equations, showing that for $n\geq 3$ each such geodesic is contained in a  3-dimensional affine subspace  and that the front tracks of these geodesics form a certain subfamily of  {\em  \Kh rods}, a class of curves introduced  in 1859 by G. Kirchhoff, generalizing the planar elastic curves of   J. Bernoulli  and L. Euler.    

\end{abstract}

\tableofcontents
\section{Introduction}

\paragraph{Bicycling geodesics.} Consider the motion of  a directed line segment of unit length in $n$-dimensional Euclidean space $\R^n$, $n\geq 2$.  As the segment moves, its end points trace a pair of trajectories,  the  {\em front} and {\em back} tracks. We consider motions satisfying the {\em no-skid} condition: {\em at each moment the line segment is tangent to the back track.} That is, if  $\x(t)$ and $\y(t)$ are the front and back tracks, respectively, and $\v(t):=\x(t)-\y(t)$ is the direction of the line segment (the `bike frame'), then  $|\v(t)|=1$ and  $\y'(t)$ is parallel to $\v(t)$ for all $t$. 

 Such a motion is called a {\em bicycle path}. For $n=2$ this is the simplest model for bicycle motion, hence the terminology, see Figure \ref{fig:bicycle}. A justification of this model is that the rear wheel of a bicycle is fixed on its frame. The same model describes hatchet planimeters. See \cite{FLT} for a survey.

 \begin{figure}[h]\centering
\def\svgwidth{.6\textwidth}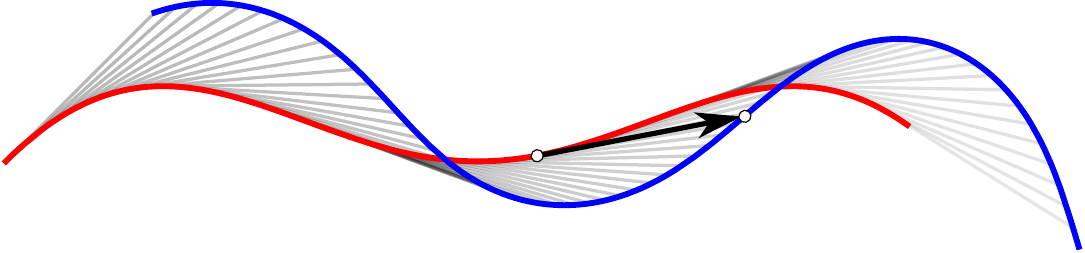
\caption{A bicycle path: 
as the line segment (``bicycle'')  moves, its end points trace the front $\x(t)$ (blue) and back $\y(t)$ (red) tracks such that  the the direction $\v(t)$ of the  line segment is tangent at each moment to the back track.} \label{fig:bicycle}
\end{figure}

We define the  {\em length} of such a path  as the (ordinary) length of its front track.  We ask: {\em  what  are the  bicycling geodesics?} These  are  paths with critical   length among  bicycle  paths  connecting two given placements of the line segment. 

The article \cite{2D} answered this question for $n=2$. The answer is that the front tracks of bicycling geodesics  are arcs of {\em non-inflectional elastic curves}, a well-known class of curves studied first by J. Bernoulli (1694) and by L. Euler (1743).

In the present article we answer this question for general $n$; it turns out that  it is enough to consider the $n=3$ case, and that the front tracks of these bicycle geodesics are  {\em  \Kh rods}, a class of curves introduced  in 1859 by G. \Kh \cite{K}, then studied extensively by many others.   
See, e.g., \cite[Chap. 5]{O} or \cite{LS} (our main reference).   Let us review this material briefly.

 \begin{figure}[h]\centering
\includegraphics[width=0.9\textwidth]{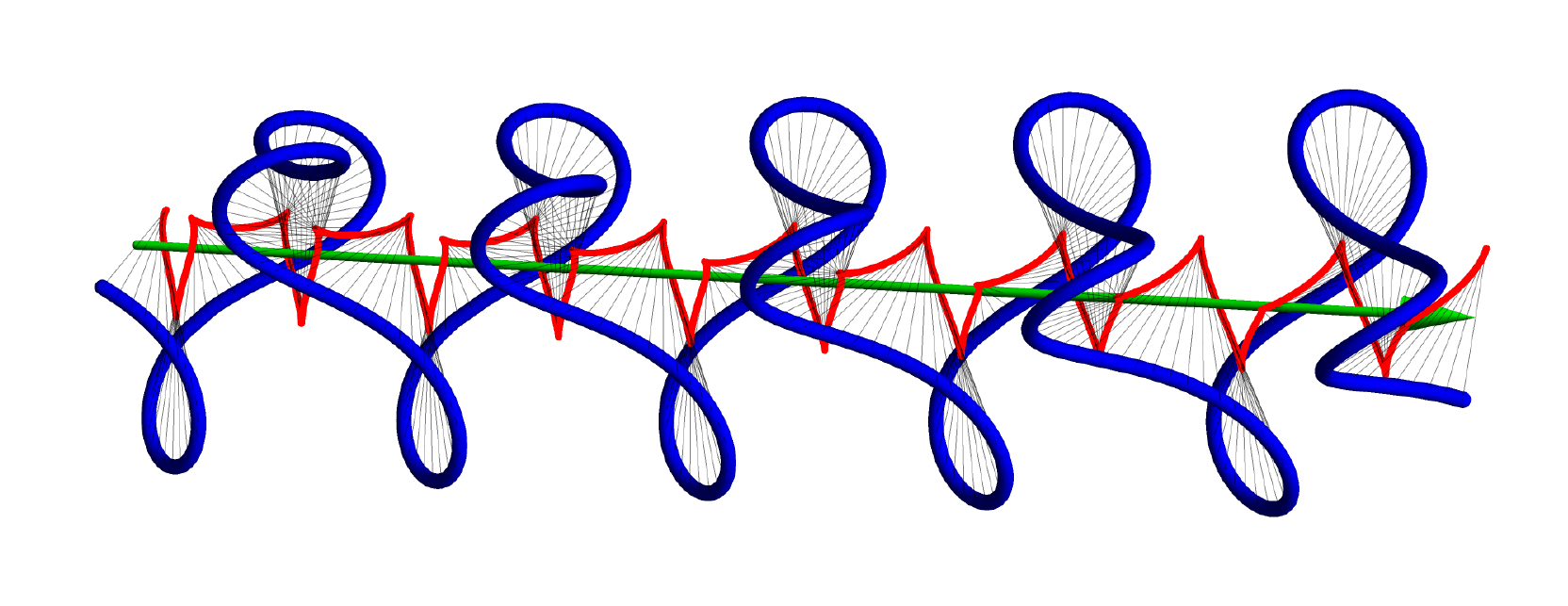}
\caption{A bicycle geodesic in $\R^3.$ The front track is blue and the back  track is red. The front track is also the trajectory of a Killing magnetic field, whose axis is marked in green.} \label{fig:geod6}
\end{figure}

 
\paragraph{\Kh rods.} These are curves in $\R^3$ which   are 
extrema of the total squared curvature    (`bending energy') among curves with fixed  end points,     total torsion, and  length. 
Accordingly, one defines the functional
$$\gamma\mapsto \int_\gamma \left(\kappa^2+\lambda_1\tau+\lambda_2\right)d s,$$
where $\lambda_1, \lambda_2$ are Lagrange multipliers, 
 and studies the associated variational  equations. The result is a 4-parameter family of  space curves  (up to rigid motions), either straight lines or curves whose curvature $\kappa$ and torsion $\tau$, as functions of arc length, satisfy
\begin{align}
& \kappa'' = \kappa\left[2a_1 + \tau(\tau -a_2)  -  \frac{\kappa^2}{2}\right],\label{eq:rods1}\\
&\kappa^2 (2 \tau-a_2)=a_3, \label{eq:rods2}\
\end{align}
where $a_1,a_2, a_3\in\R$.  The  ODE \eqref{eq:rods1} admits an `energy conservation law', 
\be\label{eq:energy}
(\kappa ')^2+\frac{1}{4} \left(\kappa^2- 2a_1\right)^2+\kappa^2 (\tau-a_2)^2=(a_4)^2, 
\ee
for some  $a_4\in\R$. 
See for example \cite[\S4]{LS}.\footnote{In \cite[\S4]{LS} there appear 5 parameters, $\lambda_1, \lambda_2, \lambda_3, c,j,$ but  $\lambda_3$ is superfluous  and can be set to $\lambda_3=1$, and the rest of the parameters are related to ours  by  $\lambda_1=a_1, \lambda_2=a_2, c=a_3, j=a_4$.}

\begin{rmrk} 
Note that equation  \eqref{eq:rods1} cannot be replaced with \eqref{eq:energy}, since there are solutions of \eqref{eq:rods2}-\eqref{eq:energy} with constant $\kappa, \tau$ which are not solutions of \eqref{eq:rods1}-\eqref{eq:rods2};  however, for solutions with non-vanishing $\kappa'$,  equations \eqref{eq:rods1} and \eqref{eq:energy} are equivalent.
\end{rmrk}

\mn 

Among \Kr, {\em elastic curves}  are those with $a_2=0$.   Planar \Kr,  i.e., those with $\tau=0$,  are planar elastic curves, satisfying  $a_2=a_3=0$. See Figure \ref{fig:elasticae}.

\begin{figure}[ht]
\centering
\includegraphics[width=\textwidth]{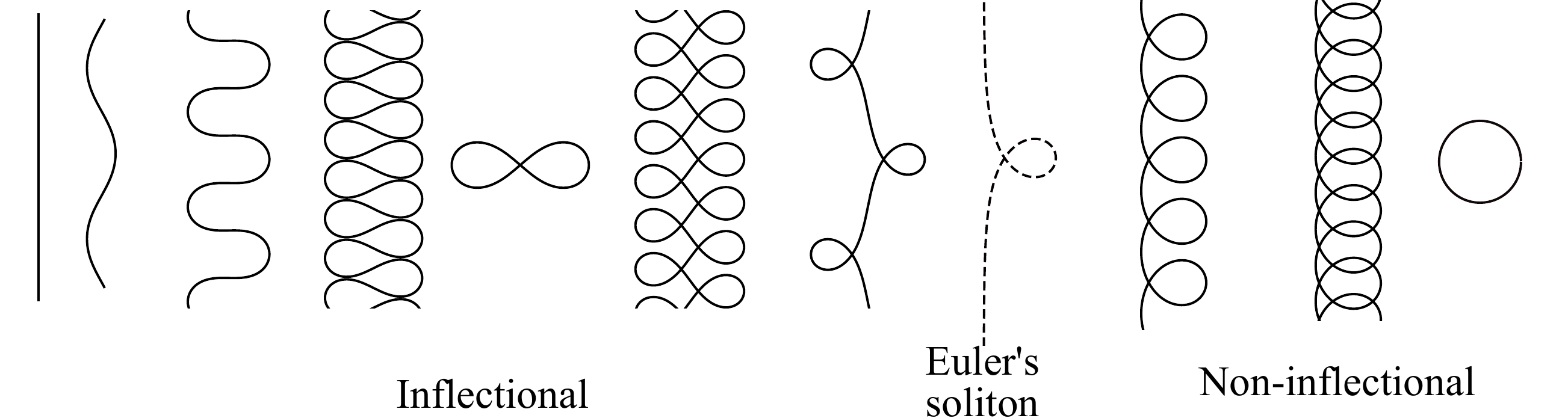} 
\caption{ The family of planar elastic curves.  }
\label{fig:elasticae}
\end{figure}

One can use equations \eqref{eq:rods1}-\eqref{eq:energy} to write  a single ODE for $u:=\kappa^2$ of the form $(u')^2=P(u)$, where $P$ is a cubic polynomial. It follows that $u,$ and therefore $\kappa$ and $\tau$, are {\em elliptic functions} (doubly periodic in the complex domain) so that  the curves themselves are {\em quasi-periodic}, i.e., $\x(t+T)=M (\x(t)),$ for some $T>0$, $M\in\Iso(\R^3)$ and all $t$.  The isometry $M$ is called the {\it monodromy} of the geodesic.

Another useful characterization of Kirchhoff rods is as the trajectories of a {\em charged particle in a Killing magnetic field}; that is, the solutions $\x(t)$  of 
\be\label{eq:mag}
\x''= \x'\times\K,\ \mbox{where } \K= (\x -\x_1)\times \p + \delta \p
\ee
for some fixed $\p,\x_1\in\R^3$, $\delta\in\R$. The vector  field $\K$ is called  `Killing', or an `infinitesimal isometry', since it generates screw-like  rigid motions about a fixed  line, the line passing through $\x_1$ in the direction of $\p$. Note that for $\delta=0$ and $\x'(0)\|\p$ the trajectory is a planar elastica. 
See  \cite{Magnetic, DM}. 

 \paragraph{The main result.}

\begin{thm}\label{thm:main}
\begin{enumerate}[{\rm (a)}]
\item  \label{thm:main_a}The front and back tracks of each bicycling geodesic  in $\R^n$, $n\geq 3$,  are contained in a 3-dimensional affine subspace. 

\item \label{thm:main_b}Front tracks of   bicycling geodesics in $\R^3$   are either straight lines or curves whose curvature and torsion functions satisfy
\begin{align*}
 &\kappa'' = \kappa\left[ \tau(\tau - b) + \frac{1 + a^2 - \kappa^2}{2}\right], \\ 
 &\kappa^2 (2 \tau-b)=b(a^2-1),
\end{align*}
and such that 
$$
 (\kappa ')^2+\frac{1}{4} \left( 1 + a^2 - \kappa^2 \right)^2+\kappa^2 (\tau-b)^2=a^2+b^2,
$$
where $a,b\in\R$. These curves comprise  a 2-parameter subfamily of \Kr, solutions to equations \eqref{eq:rods1}-\eqref{eq:energy},  with parameter values 
\be\label{eq:ab} a_1={1+a^2\over 2}, \ a_2=b , \ a_3=b(a^2-1),\ a_4=\sqrt{a^2+b^2}.\ee

\item \label{thm:main_c}A unit speed bicycle path $(\x(t), \y(t))$ in $\R^3$, that is, 
$$
|\x'(t)|=1,\ |\x(t)-\y(t)|=1,\ \y'(t)\| (\x(t)-\y(t)) \mbox{ for all } t, 
$$
with   initial conditions
$\x_0,  \x'_0, \y_0\in\R^3$, 
 is a bicycling geodesic if and only if  the front track $\x(t)$  is either a unit circle with $\y(t)$ fixed at its center, or a solution to equations
\eqref{eq:mag}
with  
$$\p\neq 0,\ \v_0\cdot(\x'_0 - \p) = 0,\ \x_1 = \y_0 +{(\x'_0\times \v_0)\times\p\over |\p|^2},\ \delta = {(\x'_0\times \v_0)\cdot\p\over |\p|^2}.$$ 
The  parameters $a,b$ of the previous item are given   by   
$$
a^2+b^2=|\p|^2,\ b=-(\x_0'\times\v_0)\cdot \p.
$$
\end{enumerate}
  \end{thm}

\begin{figure}[h]\centering
\includegraphics[width=\textwidth]{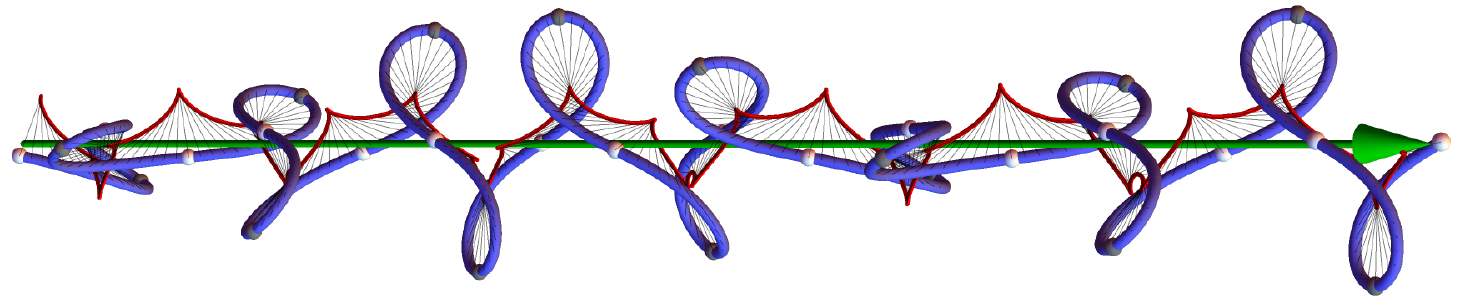}
\caption{A bicycling geodesic with front track  (blue) of constant torsion (the curvature function is that of a planar inflectional elastic curve). Points with vanishing $\kappa$ (inflection points)  are marked with light marks, maxima and minima  of $\kappa$ are marked with dark marks.  The back track is red. The green axis is the symmetry axis of the associated  monodromy and the magnetic field $\K$ of equation \eqref{eq:mag}.} \label{fig:const_torsion}
\end{figure}

\paragraph{Additional results.} A detailed description of   bicycling geodesic involves the following results proved in Section \ref{sect:add}:
\begin{itemize}
\item Front tracks of planar geodesics are the solutions of Theorem \ref{thm:main}(b) with $b=0$,  or unit circles ($a=0$). As shown in  \cite{2D}, these  front tracks are {\em non-inflectional elasticae,} see Proposition \ref{prop:planar}(a) and Figure \ref{fig:elasticae}. 

\item Solutions with $|a|=1$ correspond to front tracks with constant torsion $\tau=b/2,$ and whose  curvature is that of planar {\em inflectional elasticae}. 
See Proposition \ref{prop:planar}(b) and Figure \ref{fig:const_torsion}.

\item Each non-planar  geodesic front track comes in a `fixed size'. That is, no two such curves are related by a similarity transformation. This is unlike the planar case, where each front track, except  circle, straight line and Euler soliton,  come in two `sizes', `wide' and `narrow', see \cite[\S4]{2D}. The closest thing to it, for non-planar geodesic front tracks,  is a `torsion-shift-plus-rescaling' transformation, see Proposition    \ref{prop:inv}.

\item The only globally minimizing bicycle geodesics are the planar minimizers, i.e., those geodesics whose front tracks are either a line or an Euler's soliton. See Proposition \ref{prop:global}. 

\item There are no closed bicycling geodesics except those with circular front tracks. See Proposition  \ref{prop:closed}. 

\item Bicycle geodesics, like all \Kh rods,  can be expressed explicitly in terms of elliptic functions. See the Appendix. 

\item Back tracks of  bicycling geodesic are  determined by their front tracks.  (Exception: linear front tracks.) See Proposition \ref{prop:bt}. 

\item 
Given a rear bicycle track,  bicycle correspondence between   front tracks is the result of reversing the direction of the bicycle frame, see \cite{BLPT}. Bicycle correspondence defines  an isometric involution on the bicycle configuration space, acting on the front and back tracks of geodesics by  an isometry whose second iteration is the monodromy of the tracks involved. See Proposition \ref{prop:bc}.

\end{itemize}


\paragraph{Acknowledgments.} We are grateful to  David Singer for  a suggestion that led to Proposition \ref{prop:inv}. The referee helped us realize that Section \ref{ss:ab} was needed. In writing it, conversations  with Richard Montgomery helped us understand the mysteries of abnormal geodesics in \sR geometry and his wonderful book \cite{M}.  GB acknowledges support from CONACYT Grant A1-S-45886. ST was supported by NSF grant DMS-2005444.

\section{Proof of Theorem \ref{thm:main}}

\subsection{A \sR\ reformulation} 


We start by  reformulating bicycle paths and geodesics   in the language of \sR geometry. Our main reference here is  Chapter 5 of the book \cite{M}.

Denote by $\v:=\x-\y\in S^{n-1}$ the frame direction and by 
$$Q:=\{(\x,\v)\,|\, |\v|=1\}=\R^n \times S^{n-1}$$ the bicycling configuration space. The no-skid condition defines an 
$n$-distribution $\D$ on  $Q,$
 that is, a rank $n$ sub-bundle  $\D\subset TQ$, so that  
bicycle paths are   curves   in $Q$ tangent everywhere  to $\D$.

\begin{lemma} \label{lemma:distr}
$\D$ consists of vectors $(\x', \v')\in T_{(\x, \v)}Q$ satisfying
\be\label{eq:noskid}\v'=\x'-(\x'\cdot\v)\v.\ee
\end{lemma}
 
 The proof appeared before, e.g., in  Proposition 2.1 of \cite{BLPT},  or Lemma 4.1 of \cite{2D}. Since the proof is quite short, and our notation here differs slightly from that of these references, we reproduce it here. 
\begin{proof} Let $(\x(t), \v(t))$ be a path in $Q$ and $ \x'_\| :=(\x'\cdot\v)\v$ the orthogonal projection of the front track velocity on the frame direction. The   no-skid condition is then $\y'=\x'_\|.$ From  $\v=\x-\y$ it follows that  $\y'=\x'_\|$ is equivalent to $\v' =\x'-\x'_\|,$    which is equation \eqref{eq:noskid}. 
\end{proof}

 
 The {\em \sR length} of  a vector $(\x', \v')\in \D$ is, by definition,  the length $|\x'|$ of its projection to $\R^n$.  This defines a \sR\ structure $(Q,\D, g)$, where  $g$ is  a positive definite quadratic form on $\D$, the restriction to $\D$ of the pull-back to $Q$ of the standard riemannian metric on $\R^n$ under the `front wheel' projection $Q\to \R^n$, $(\x,\v)\mapsto \x.$ In this language, bicycle geodesics are the geodesics of $(Q,\D, g)$, that is,  curves in $Q$ tangent to $\D$ whose length between any two fixed points on them is critical among curves in $Q$ tangent to $\D$ whose end points are these fixed points.  
 In fact, as in the Riemannian case,   if the geodesic arc  is sufficiently short then it is 
minimizing between its endpoints.  See Theorem 1.14 on page 9 of \cite{M}. 

 Now in general, \sR geodesics  are either  {\em normal} or {\em abnormal}. Normal geodesics always exist and are given by solutions of an analog of the usual geodesic equations in Riemannian geometry. We shall next derive these equations for bicycle geodesics in Lemma \ref{lm:ham1} below. Abnormal \sR geodesics do not have a Riemannian analog and are harder to pin down. We shall later show, in section \ref{ss:ab},  that abnormal bicycle geodesics in fact do not exist, i.e., all bicycle geodesics satisfy the geodesic equations, see Corollary  \ref{cor:ab}. 

 \subsection{The bicycle  geodesics equations} 
The (normal) geodesic equations on a \sR\ manifold $(Q,\D,g)$ are derived via a  Hamiltonian formalism, as follows. 
 
 One fixes  an orthonormal frame  $\xi_i$ for $\D$ and  let $P_i:T^*Q\to \R$ be the associated   fiber-wise linear momentum functions, $$P_i(\alpha):=\alpha(\xi_i),\ \alpha\in T^*Q.$$
One  forms the Hamiltonian $H:=(1/2)\sum_i(P_i)^2$ on $T^*Q$ and the associated Hamiltonian vector field $X$ (with respect to the standard symplectic structure on $T^*Q$). The  \sR normal geodesics are then the  projection  to $Q$ of the   integral curves of $X$.  See Definition 1.13 on page 8 of \cite{M}. We shall now follow this recipe in our case. 

Let $\partial_{x_i}$ be the standard basis in $\R^n$ and $\bfe=\sum_{j=1}^n v_j\partial_{v_j}$ be the unit normal along $S^{n-1}\subset \R^n.$ Then, by   Lemma \ref{lemma:distr}, the vectors
\be\label{eq:xi}
\xi_i:=\partial_{x_i} + \partial_{v_i} - v_i \bfe,\ i=1,\ldots, n,
\ee
form an orthonormal basis of $\D$. 

Using the Euclidean structure on $\R^n$  we identify 
$$
T^*(\R^n\times\R^n)=T(\R^n\times\R^n),\ T^*Q=TQ,
$$
 so that $T^*Q\subset T^*(\R^n\times\R^n)$ is a symplectic submanifold.  Let $p_i,r_i$ be the momenta coordinates on $T^*(\R^n\times\R^n)$ dual to $x_i, v_i$; that is, if $\alpha\in T^*(\R^n\times\R^n)$ then $p_i(\alpha):=\alpha(\partial_{x_i}),\  r_i(\alpha):=\alpha(\partial_{v_i}).$ We shall use the same letters $x_i,v_i,p_i,r_i$ to denote the restriction of these functions to $T^*Q$. 

\mn {\bf Notation.} We shall use a vector notation throughout: 
\begin{equation*}
\begin{split}
{\bf a}=(a_1, \ldots, a_n),\ {\bf a}\cdot{\bf b}=\sum_{i=1}^n a_ib_i,\ {\bf a}\,\partial_\x=\sum_{i=1}^n a_i\partial_{x_i},\\
 \omega=d\x\wedge d\p+d\v\wedge d\r=\sum_{i=1}^n dx_i\wedge dp_i+dv_i\wedge dr_i, 
\end{split}
\end{equation*}
 etc.

\begin{lemma}\label{lemma:rv}$T^*Q\subset T^*(\R^n\times\R^n)$ is given, in the coordinates $\x,\v,\p,\r$,  by $$\v\cdot\v=1,\quad  \r\cdot\v=0.$$
\end{lemma}
\begin{proof}Let $(\x,\v)\in Q$, so that $\v\cdot \v=1$, and 
$$\alpha=\p d\x + \r d\v\in T^*_{(\x,\v)}(\R^n\times\R^n). 
$$
 Then $\alpha\in T^*Q$ if and only if the corresponding vector, 
$$X=\p\partial_\x+\r\partial_\v\in T_{(\x,\v)}(\R^n\times\R^n),$$  satisfies $X(\v\cdot \v)=0$. That is, $0=\v d\v (X)= 
\r\cdot\v=0.$
\end{proof}

\begin{lemma} \label{lm:ham}
$P_i=p_i+r_i$ on $T^*Q$. Thus  
\be\label{eq:H}
H={1\over 2}|\p+\r|^2=
{1\over 2}\sum_{i=1}^n(p_i+r_i)^2.
\ee
\end{lemma}
\begin{proof} Let $\alpha=\p d\x + \r d\v\in T^*Q$. By equation \eqref{eq:xi},  
$P_i(\alpha):=\alpha(\xi_i)=p_i+r_i + v_i  (\v\cdot \r).$ By the previous lemma, the last term vanishes. 
\end{proof}

\begin{lemma} \label{lm:ham1}
The Hamiltonian equations on $T^*Q$, corresponding to the Hamiltonian \eqref{eq:H}, are 
\be\label{eq:ham}
 \begin{alignedat}{1}
  &\x'=\p+\r,\\
  &\v'=\p+\r-(\v\cdot\p)\v,\\
 &\p'=0,\\ 
 &\r'=(\v\cdot \p)\r - [\r\cdot (\r+\p)] \v,
 \end{alignedat}
\ee
with $\v\cdot \v=1,$ $\r\cdot \v=0$.
 \end{lemma}

\begin{proof}
Let $X$ be the Hamiltonian vector field on $T^*Q$,
$$
X=\x'\partial_\x + \v' \partial_\v + \p' \partial_\p + \r' \partial_\r,
$$
where $\x', \v', \p', \r'$ are unknown vectors. We have
\begin{equation} \label{eq:def}
 dH=i_X \omega,
\end{equation}
where  $\omega = d\x\wedge d\p+ d\v\wedge d\r$ is the symplectic form.  

Now  
$$i_X \omega=\x'd\p+\v'd\r-\p'd\x-\r'd\v$$ and,  by equation \eqref{eq:H}, 
$$dH=(\p+\r) (d\p+d\r).$$  Note that  equation \eqref{eq:def}   is an equality between 1-forms on $T^*Q$, the restrictions of both sides of   \eqref{eq:def}  to $T^*Q$. By Lemma \ref{lemma:rv}, the kernel of this restriction is spanned by  $\v d\v, \r d\v+\v d\r$. Thus, equation \eqref{eq:def} amounts to the existence of functions $\lambda, \mu$ on $T^*Q$ such that 
\begin{align*}
\begin{split}
(\p+\r) (d\p+d\r)=&\ \x'd\p+\v'd\r-\p'd\x-\r'd\v\\
&\qquad +\lambda \v d\v+\mu(\r d\v+\v d\r).
\end{split}
\end{align*}
Equating coefficients, we obtain 
\be \label{eq:ham1}
\x'=\p+\r,\ \v'=\p+\r-\mu \v,\ \p'=0,\ \r'=\lambda \v + \mu \r.
\ee
Furthermore, since  $X$ is a vector field on $T^*Q$, $\v d\v, \r d\v+\v d\r$ vanish on $X$, hence 
$$ \v\cdot \v' =0,\  \v'\cdot \r + \v\cdot\r'=0.
$$
Dotting the second equation\ of \eqref{eq:ham1} with $\v$, one obtains  
$$
0= \v\cdot\v'=\v\cdot(\p+\r-\mu \v)=\v\cdot \p
-\mu,
$$
hence $\mu=\v\cdot\p.$ 
Dotting the 4th equation\ of \eqref{eq:ham1} with $\v$ one obtains   $\v\cdot\r'=\v\cdot ( \lambda\v+\mu\r)=\lambda,$ hence 
$$
\lambda=- \v'\cdot \r=-(\p+\r-\mu\v)\cdot \r=-(\p+\r)\cdot \r.
$$
 Substituting these values of $\lambda,\mu$ in equations \eqref{eq:ham1}, we obtain equations (\ref{eq:ham}).
\end{proof}

\begin{lemma}\label{lemma:p} The tri-vector $\p \wedge\v\wedge \x'$ is constant along solutions of equations \eqref{eq:ham}. 
\end{lemma}

\begin{proof}Using equations \eqref{eq:ham}, 
\begin{equation*}
\begin{split}
(\p \wedge\v\wedge \x')'&=
(\p \wedge\v\wedge \r)'=\p \wedge\v'\wedge \r+\p \wedge\v\wedge \r'\\
&=-(\v\cdot\p)\p \wedge\v\wedge \r+(\v\cdot\p)\p \wedge\v\wedge  \r=0,
\end{split}
\end{equation*}
as needed.
\end{proof}

\begin{cor} [Part (a) of Theorem \ref{thm:main}] \label{cor:3d} For any normal bicycling geodesic (projection to $Q$ of a solution of equations  \eqref{eq:ham}), the front and back tracks $\x(t), \v(t)$  are contained in the affine space passing through $\x_0$, parallel to the linear subspace spanned by  $\x'_0, \v_0, \p.$ \end{cor}

\begin{rmrk}
One can give also a geometric argument  for Corollary \ref{cor:3d}.  Let $(\x(t),\y(t)),\ t\in[t_0,t_1],$ be a unique minimizing geodesic (we assume that the interval $[t_0,t_1]$ is small enough). Generically, the points  $\x(t_0), \y(t_0), \x(t_1), \y(t_1)$ span a 3-dimensional affine space, and the reflection in this subspace induces an isometry of the configuration space $Q$. If the geodesic is not contained in this 3-space, then its reflection is another minimizing geodesic, contradicting its uniqueness.
\end{rmrk}

Another   consequence of equations  \eqref{eq:ham} is 
\begin{cor}
\begin{enumerate}[{\rm (a)}]
\item  $\x'\cdot \x'=2H$ is constant along solutions of equations \eqref{eq:ham}. 
\item  If $(\x(t), \v(t), \p, \r(t))$ is a solution then so 
is $(\x(\lambda t), \v(\lambda t), \lambda\p, \lambda\r(\lambda t)$ for all $\lambda\neq 0.$
\end{enumerate}
\end{cor}

We now proceed to proving parts (b), (c) of Theorem \ref{thm:main}. By the last corollaries, we shall assume henceforth  that $n=3$ and $|\x'(t)|=1$ (arc length parametrization of the front track).

\begin{lemma}\label{lemma:bc} For any solution of equations \eqref{eq:ham}:
\begin{enumerate}[{\rm (i)}]
\item $b:=\p\cdot( \v\times \x')$ is constant.

\item $|b|\leq  |\p|,$ with equality if and only if  the front track is a unit circle, with the back wheel staying fixed at the center of the circle. 

\item If a geodesic  has a front track which is a straight line, $\x''=0$, then  $b=0$, $|\p|=1$.  (The converse is not true, because the Euler soliton as the front track also corresponds to these values, see Proposition \ref{prop:planar} below). 

\item  Equations \eqref{eq:ham} are equivalent to 
\be\label{eq:Tv3}
\T'=\T\times\K, \ \v'=\T-(\v\cdot \p) \v,\quad  
\ee
where 
$$
\T=\x',\ \K=\r\times\v=(\T-\p)\times\v,\ \p=const, |\T|=|\v|=1,\  \r\cdot\v=0.
$$
\end{enumerate}
\end{lemma}

\begin{proof} (i) follows  from  Lemma   \ref{lemma:p}. 

\mn   (ii)  By the Cauchy-Schwartz inequality, $|b|=| \p\cdot( \v\times \x')|\leq |\p| | \v| | \x'|  = |\p|,$   with  equality if  either $\p = 0$
 or $\p\ne 0$ and $\x', \v, \p$ are pairwise orthogonal. In the first case $\x' = \r$, hence  is perpendicular to $\v$, and in both cases, setting $\T= \x'$, equations \eqref{eq:ham} give 
$\T'' =- \T ,\ \v=-\T',$
and claim  (ii) follows. 

\mn (iii) The first equation of \eqref{eq:ham} implies that $\r=\x'-\p=const,$  thus, by the  4th equation, $(\v\cdot\p)\r=(\r\cdot \x')\v.$  Dotting with $\r$, we get $(\v\cdot\p)|\r|^2=0,$ so either $\v\cdot\p=0$ or $\r=0$. If  $\v\cdot\p=0$ then  by the 3rd equation, $\v'=\x', $ so $\v=\v_0+t\x'$, which is impossible since $|\v|=|\x'|=1$. Hence $\r=0$, so $\p=\x'$, which implies that  $b=\p\cdot( \v\times \x')=\x'\cdot( \v\times \x')=0$ and  $|\p|=|\x'|=1$, as needed.

\mn (iv)   The 3rd equation  of \eqref{eq:ham} gives $\T'=(\v\cdot \p)\r - (\r\cdot\T)\v.$ By the vector identity 
\be\label{eq:vi}
{\bf a}\times({\bf b}\times{\bf c})=({\bf a}\cdot{\bf c}){\bf b}- ({\bf a}\cdot{\bf b}){\bf c},
\ee
 this is equivalent to  $\T'=\T\times(\r\times\v).$ 
The second equation of \eqref{eq:Tv3} is immediate from the first and second equations of \eqref{eq:ham}.
\end{proof}

\begin{rmrk}
The  front track of a  bicycling geodesic with $\p = 0$  is thus  a unit circle with the back track fixed at its center. From here on, unless otherwise mentioned, we will only consider bicycling geodesics with $\p\neq  0$.
\end{rmrk}

\begin{prop}[Part \ref{thm:main_c} of Theorem \ref{thm:main}] 
A unit speed bicycle path $(\x(t), \v(t))$ in $\R^3$ 
with  initial conditions
$\x_0, \x'_0, \v_0$
 is a bicycling geodesic  (a solution to equations \eqref{eq:Tv3}) with $\p\neq 0$  if and only if  $\x(t)$  is a solution to 
\be\label{eq:mag_geod}
\x''=\x'\times \K,
\ \mbox{ where }\K  =(\x-\x_1)\times\p+ \delta\p, 
\ee
and 
$$ \v_0\cdot(\x'_0 - \p) = 0,\ \x_1 = \y_0 +{(\x'_0\times \v_0)\times\p\over |\p|^2},\ \delta = {(\x'_0\times \v_0)\cdot\p\over |\p|^2}=-{b\over |\p|^2},$$ 
 where $\y_0 = \x_0 - \v_0$ is the initial back track position. See Figure \ref{fig:vectors}.

\end{prop}

\begin{figure}
\center
\def\svgwidth{.3\textwidth}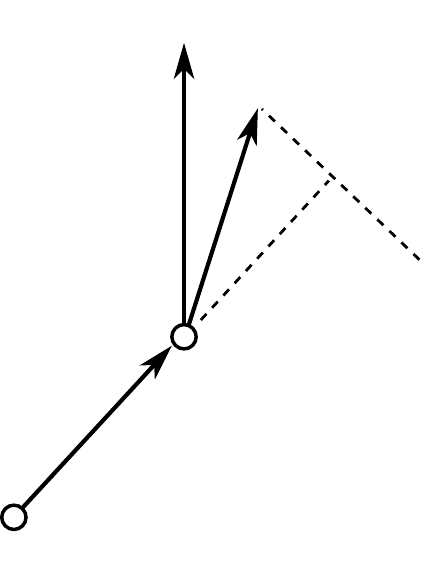
\caption{}\label{fig:vectors}
\end{figure}

\begin{proof} 
From equations \eqref{eq:Tv3}, 
\be\label{eq:B}
\x''=\x'\times\K,
\ee where $\K:=\r\times\v$ and  $\r=\x'-\p.$ By  equations \eqref{eq:ham},
\begin{align} \label{eq:magDer}
\begin{split}
\K' &=\r'\times\v + \r\times\v' = (\v\cdot\p)\r\times\v +   \r\times[\p-(\v\cdot\p)\v]=
\r\times\p\\
&=(\x\times\p)'.
\end{split}
\end{align}
Consequently, 
$$\K -\x\times\p=\r_0\times\v_0 -\x_0\times\p,$$
or
\be\label{eq:rv}
 \K = (\x - \y_0)\times\p + \x_0' \times\v_0.
 \ee
Since $\p\ne 0$, we may decompose orthogonally
\be
 \x_0' \times\v_0= \p\times {\bf a} + \delta \p,
\ee
for some ${\bf a}\perp\p$ and $\delta\in\R$. Then   
$$ (\x_0'\times\v_0)\times\p= |\p|^2{\bf a}, \ (\x_0' \times\v_0)\cdot\p=|\p|^2\delta,
$$
hence 
\be\label{eq:adel}{\bf a}= {(\x_0'\times\v_0)\times\p\over |\p|^2},\ 
\delta={(\x_0' \times\v_0)\cdot\p\over |\p|^2}=-{b\over |\p|^2}.
\ee 
Equations \eqref{eq:mag_geod} now follow from \eqref{eq:B}-\eqref{eq:adel}  by taking $\x_1 := \y_0 + {\bf a}$.

Conversely, the magnetic field $\K$ and initial conditions for its trajectory are defined using only the initial conditions for the bicycling geodesic's front track. Since they satisfy the same equations of motion (equation \eqref{eq:Tv3}), such magnetic field trajectories coincide with the bicycling geodesics.
\end{proof}

Let $\T,\N,\B$ be the Frenet-Serret frame along a non-linear front track $\x(t)$, and $\kappa,\tau$ the curvature and torsion functions, respectively. That is, 
\be\label{eq:FS}
 \x'=\T,\ \T'=\kappa\N,\ \N'=-\kappa\T+\tau\B,\ \B'=-\tau\N,
\ee
where $\B=\T\times\N$ (the Frenet-Serret equations).

\begin{rmrk}
\label{rmrk:FS}
The Frenet-Serret frame is  usually defined via   formulas \eqref{eq:FS} along a regular curve $\x(t)$ in $\R^3$, parametrized by arc length, with {\em non-vanishing acceleration} $\x''$,   by adding  the condition  $\kappa>0$.  If one does not add the last condition, then the frame is well defined only up to the involution 
$$(\T, \N,\B,\kappa,\tau)\mapsto (\T, -\N,-\B,-\kappa,\tau).$$   
For analytic curves, as is our case (the right hand side of equations \eqref{eq:ham} are quadratic polynomials), $\x''$ either vanishes identically, in which case it is a line, or vanishes at isolated points, 
the {\em inflection points} of the curve. 

 In the latter case, by looking at the Taylor series of $\x(t)$ around an inflection point, say $\x(0)$, one  sees that that the Frenet-Serret frame extends analytically to these points,
so equations \eqref{eq:FS} still hold, 
but $\kappa$ may change sign at the inflection point. For example, if $\x'''(0)\neq 0$, then $\kappa'(0)\neq 0$. 

That is, any analytic non-linear  regular curve $\x(t)$ admits exactly {\em two} Frenet-Serret frames,  $(\T,\N,\B,\kappa,\tau)$ 
and $(\T,-\N,-\B,-\kappa,\tau)$, both satisfying the Frenet-Serret equations \eqref{eq:FS}, but if the $\kappa$ has variable sign there   is no natural way to choose one of the frames.  
This situation actually occurs for some of the solutions of equations \eqref{eq:ham}, 
as we shall see later (the constant torsion solutions, see Proposition  \ref{prop:planar} below). 

In summary, in what follows, whenever we mention  ``the  Frenet-Serret frame'', we implicitly  refer to either choice of  these frames in case $\kappa$ has a variable sign. One can check  that all equations involving the frame are invariant under the involution $(\T, \N,\B,\kappa,\tau)\mapsto (\T, -\N,-\B,-\kappa,\tau)$. 
For more details on inflection points of analytic space curves see \cite{N}, or pages 41-43 of \cite{G}.
\end{rmrk}

\begin{lemma}For any non-linear geodesic front track,
\be\label{eq:pconst} \p={1+a^2-\kappa^2\over 2}\T-\kappa'\N-\kappa(\tau-b)\B=const., 
\ee
where  $a^2:=|\p|^2-b^2$  (see  Lemma \ref{lemma:bc}).   

\end{lemma}
\begin{proof} 
We   compute each of the coefficients $\p\cdot\T, \p\cdot\N, \p\cdot\B$.
 Recall from \eqref{eq:Tv3} and \eqref{eq:magDer} that for $\K := \r\times\v$ we have
\be\label{eq:TK}
\T' =\T\times\K ,\quad \K' =\T\times\p.
\ee

Now, since $\r = \T - \p$ and $b = \p\cdot (\v\times\T)$, we have $\T\cdot \K = \T\cdot (\r\times \v) = -b$. Using again the vector identity \eqref{eq:vi}, 
\be \label{eq:rel}
\kappa \B = \T\times \T' = \T\times (\T\times\K) = -b\T - \K.
\ee
It follows that $|\K|^2 = b^2 + \kappa^2$, and since $\r\cdot \v  = 0$, $|\v| = 1$, we have:
\[ b^2 + \kappa^2 = |\K|^2 = |\r|^2 = |\T - \p|^2 = 1 - 2\T\cdot \p + |\p|^2,\]
yielding the expression for the first component, $\T\cdot \p$.

For the remaining components, using equations \eqref{eq:TK} and   \eqref{eq:vi}, one has
\begin{align*}
\T'' &= (\T\times \K)' = (\T\times\K)\times \K + \T\times (\T\times \p)\\
&= \left( \T\cdot\p  - |\K|^2\right)\T - b\K - \p.
\end{align*}
On the other hand, by the Frenet-Serret equations, 
\[  \T'' = - \kappa^2 \T + \kappa' \N + \kappa\tau \B.\]

Dotting these two expressions for $\T''$ with $\N$, we obtain $\kappa' = -b\K\cdot \N - \p\cdot \N$, while from \eqref{eq:rel} we have $  \K\cdot \N = 0$, so that $\p\cdot \N = - \kappa'$, as needed.

Dotting the  two expressions for $\T''$ with $\B$ we obtain 
$\kappa\tau = -b\K\cdot \B  -   \p\cdot \B$, while from \eqref{eq:rel} we have $ \K\cdot \B = -\kappa$, so that $\p\cdot  \B  = - \kappa(\tau - b)$.
\end{proof}

Now we prove the second statement of Theorem \ref{thm:main}.

\begin{prop} Non-linear geodesic front tracks   in $\R^3$ (solutions to equations \eqref{eq:ham}) are curves whose curvature and torsion functions satisfy 
\begin{align}
& \kappa'' = \kappa\left(  \tau(\tau -b) + \frac{1 + a^2 - \kappa^2}{2}\right),\label{eq:geods1}\\
&\kappa^2 (2 \tau-b)=b(a^2-1),\label{eq:geods2}
\end{align}
such that   
\be\label{eq:geods3}
(\kappa ')^2+\frac{1}{4} \left( 1 + a^2 - \kappa^2 \right)^2+\kappa^2 (\tau-b)^2=a^2+b^2, 
\ee
where $a,b\in\R$.
\end{prop}

\begin{proof}Equation \eqref{eq:geods3} is obtained by taking the norm square of both sides of \eqref{eq:pconst}. Equation \eqref{eq:geods2} is obtained by dotting \eqref{eq:rel} with $\p$,
\[ b = -\K\cdot\p  =   b \T\cdot \p   +\kappa  \B\cdot  \p, \]
then   substituting  the values of $\T\cdot \p ,\    \B\cdot \p $ from equation \eqref{eq:pconst}. Equation  \eqref{eq:geods1} follows by differentiating \eqref{eq:pconst}:
\[ 0 = \left( -\kappa'' + \kappa\left[\tau(\tau- b) + \frac{1 + a^2 - \kappa^2}{2}\right]\right) \N - \left(\kappa'\tau + \left[\kappa(\tau - b)\right]'\right)\B.\]
The vanishing of the $\N$ component gives equation \eqref{eq:geods1}. 

Note that  the vanishing of the $\B$ component in the last equation does not give new information:  multiplying the $\B$ component  by $-2\kappa$, one obtains  the derivative of $ \kappa^2(2\tau - b)$,  which vanishes by  equation \eqref{eq:geods2}.
\end{proof}

\subsection{Bicycle geodesics are normal}\label{ss:ab}
%

Here we show that all bicycle geodesics in $Q=\R^n\times S^{n-1}$ are normal, that is,   the projections to $Q$ of solutions to equations \eqref{eq:ham}. Our main reference here is section 5.3 of \cite{M}.

For a general \sR\ structure $(Q,\D,g)$, abnormal geodesics are  {\em singular} curves of  $(Q,\D)$ (the  definition of singular curves does not involve  the \sR metric $g$, see below).  The converse is not true:  a singular geodesic may happen to be normal \cite[\S5.3.3]{M}. Thus, in order to show that all bicycle geodesics are normal, we will first find the  singular curves of $(Q,\D)$, then show that all geodesics among them are normal, i.e., can be lifted to parametrized curves in $T^*Q$ satisfying  equations \eqref{eq:ham}. 

Singular curves of  $(Q,\D)$ are defined by considering first the  annihilator $\D^0\subset T^*Q$, that is,  the set of covectors vanishing on $\D$. 
A {\em characteristic curve} of   $\D^0$ is a curve in $\D^0$ which does not intersect the zero section of $\D^0\to Q$ and whose tangent is in the kernel of  the restriction of the canonical symplectic form of $T^*Q$ to $\D^0$. 
A {\em singular curve} of $(Q,\D)$ is the projection to $Q$ of a characteristic curve of $\D^0$. 

The case  $n=2$ is special, since in this case   $\D$, defined by the non-skid condition \eqref{eq:noskid}, is contact, which implies that $\D^0\subset T^*Q$ is symplectic (see \cite[\S4.1]{2D} and  the example at the top of page 59 of \cite{M}). Hence there are no characteristics and  singular curves for $n=2$ so all geodesics are automatically normal. We thus assume henceforth that $n>2$.

\begin{prop}\label{prop:ab}
Singular bicycle paths   consist of curves $(\x(t),\v(t))$ in $Q$  in which the back wheel $\x(t)-\v(t)\in\R^n$ is  fixed and $\v(t)$ moves in $S^{n-1}$ perpendicular to some fixed $\p\neq0$, $\p\cdot\v(t)=0.$
\end{prop}

\begin{proof} We first determine  $\D^0\subset T^*Q$. Recall from Lemma \ref{lemma:rv} that $T^*Q$ is given in the canonical coordinates $\x,\v, \p,\r$ on $T^*(\R^n\times\R^n)$ by $\v\cdot\v=1,\ \r\cdot \v=0$. 

\begin{lemma}\label{lm:an}
$\D^0$ is a $(3n-2)$-dimensional  submanifold of $T^*Q$, given by  $\v\cdot\v=1,\ \r\cdot \v=0,\ \r+\p=0$. 
\end{lemma}
\begin{proof}
Let 
$$\alpha=\p d\x +\r d\v\in T^*Q, \ X=\x'\partial_\x+\v'\partial_\v\in \D.$$
 By  Lemma \ref{lemma:distr}, $\v'=\x'-(\x'\cdot\v)\v$. Thus, using that $\r\cdot \v=0$ on $T^*Q$,  $\alpha\in \D^0$  if 
and only if 
$$0=\alpha(X)=\p\cdot\x'+\r\cdot\left[\x'-(\x'\cdot\v)\v)\right]=(\p+\r)\cdot\x'$$
 for all $\x'\in\R^n$. That is, $\p+\r=0$, as claimed.

  To prove that $\dim(\D^0)=3n-2$,  using the implicit function theorem, we show  that $(1,0,{\bf 0})\in \R\times\R\times\R^n$ is a regular value of 
  $$(\x,\v,\p, \r)\mapsto (|\v|^2, \r\cdot\v, \r+\p).$$
   That is, for a fixed
   $(\x,\v,\p, \r)\in\D^0$, $a,b\in\R$, ${\bf c}\in\R^n,$  one needs to solve $$\v\cdot \v'=a,\ \r\cdot\v'+\r'\cdot\v=b,\ \r'+\p'={\bf c}
   $$ 
   for $\v', \p', \r'\in\R^n.$ These equations are solved by $\v'=a\v, \r'=b\v, \p'={\bf c}-b\v.$ 
\end{proof}

\begin{lemma} \label{lm:rest} 
$$\omega:=d\x\wedge d\p +d\v\wedge d\r\equiv (d\x-d\v)\wedge d\p\not\equiv 0
  \mod \D^0,$$
where equivalence of forms mod $\D^0$  means the equality of their restriction to $\D^0$. 
\end{lemma}
\begin{proof}The first congruence follows from the equality $\r+\p=0$ on $\D^0$, proved in the  last lemma. For $n>2$ we have, by the same lemma,  
$$\dim(\D^0)=3n-2>2n={1\over 2}\dim\left[T^*(\R^n\times\R^n)\right],
$$ so the restriction to $\D^0$ of $\omega$ (the canonical symplectic form of $T^*(\R^n\times\R^n)$)  is non-vanishing.  
\end{proof}

Next let 
$$Y=\x'\partial_\x+\v'\partial_\v+\p'\partial_\p+\r'\partial_\r
$$
 be a vector tangent to $\D^0$. 
\begin{lemma}$Y$ is in the kernel of the restriction of $\omega$ to $\D^0$, $i_Y\omega\equiv 0\mod\D^0$,  if and only if $\p'=\r'=\x'-\v'=0.$ 
\end{lemma}
\begin{proof} By Lemma  \ref{lm:an}, 
$$\D^0=\{\v\cdot\v=1, \r\cdot\v=0, \p+\r=0\},
$$
 hence 
tangency of $Y$ to $\D^0$ amounts to
$$\v\cdot \v'=\r\cdot \v'+\v \cdot\r'=0,\ \p'+\r'=0, 
$$ 
thus
$$Y=\x'\partial_\x+\v'\partial_\v+\p'(\partial_\p-\partial_\r).
$$
It follows that 
\begin{align*}
i_Y\omega&\equiv i_Y\left[ (d\x-d\v)\wedge d\p\right]=  (\x'-\v')d\p- \p'(d\x-d\v) \mod\D^0.
\end{align*}
By Lemma \ref{lm:rest}, $d\p\wedge(d\x-d\v)\not\equiv0\mod\D^0.$ It follows that the restrictions of $d\p, d\x-d\v$ to $\D^0$ are linearly independent, hence the vanishing of $i_Y\omega\mod\D^0$ is equivalent to the vanishing of $\x'-\v', \p'.$  Since $\r'=-\p'$ for vectors tangent to $\D^0$, $\r'$ vanishes as well. 
\end{proof}

We can now complete  the proof of Proposition \ref{prop:ab}. From the last lemma follows that characteristics of $\D^0$   are  curves $(\x(t), \v(t), \p(t), \r(t))$ where the back track $\x(t)-\v(t)$ is fixed, $|\v(t)|=1$, $\r(t)\cdot\v(t)=0$  and $\p(t)=-\r(t)$ is a non-zero constant vector. This projects to the stated  curve in $Q$. 
\end{proof}

\begin{rmrk} Proposition \ref{prop:ab} is a special case of the following. Let $M$ be an $n$-dimensional manifold and $Q=S(TM)$ the spherized tangent bundle, consisting of pairs $(y,\ell)$, where $y\in M$ and $\ell$ is an oriented 1-dimensional subspace of $T_yM$. Define a   rank $n$ distribution $\D\subset TQ$ whose integral curves are given by  trajectories $(y(t), \ell(t))$ such that $\dot y(t)\in \ell(t)$ for all $t$.  (One can think of $Q$ as the configuration space for bicycling on $M$, where $y$ is the back wheel placement, $\ell$ is the frame direction and $\D$ is the no-skid condition.) Note that the trajectories for which the back wheel is fixed, $y(t)=const.$, satisfy this condition trivially. Singular curves of $(Q,\D)$ are then given by trajectories $(y(t), \ell(t))$ such that  $y(t)=y$ is fixed and $\ell(t)$ varies in some fixed codimension 1 subspace of $T_yM$. Proposition \ref{prop:ab} is the case of $M=\R^n$ and actually implies the more general case since it is a local statement. 
\end{rmrk}
\begin{cor}\label{cor:ab}All bicycle geodesics are normal. 
\end{cor}
\begin{proof}By Proposition \ref{prop:ab}, the length of a singular bicycle path  $(\x(t), \v(t))$ is given by the length of the curve traced by  $\v(t)$  on $S^{n-1}$. This length is critical  when $\v(t)$ traces a spherical geodesic, i.e., an arc of a great circle  on $S^{n-1}$  (the intersection of a 2-dimensional subspace of $\R^n$ with $S^{n-1}$). To show that this singular bicycle geodesic is normal we first reparametrize it by arc length,  $|\v'(t)|=1,$ then  lift it  to a solution of equations \eqref{eq:ham}. Let 
$\p\neq 0$ be a vector perpendicular to the 2-plane spanned by $\v(t)$
and $\r(t):=\v'(t)-\p$.  Then one can easily check that this defines a solution to equations  \eqref{eq:ham}.  (Note that this lift is {\em not} a  characteristic of $\D^0$, since for a characteristic $\r=-\p$ is constant.)
\end{proof}

\section{Additional results} \label{sect:add}

\subsection{More about geodesic front tracks}

In   Theorem \ref{thm:main}  we described front tracks of  bicycling geodesics  as a  subfamily of  \Kr, parametrized by the two parameters $a,b$  in equations \eqref{eq:geods1}-\eqref{eq:geods3}. Here  we give more information on these curves.

Clearly, since $a$ appears in equations \eqref{eq:geods1}-\eqref{eq:geods3} only through $a^2$, it is enough to restrict to $a\geq 0$. 
Regarding   $b$, we observe  the following. 

\begin{lemma}\label{prop:ref}
Reflection with respect to a plane or a point  in $\R^3$ transforms a bicycling geodesic to another bicycling geodesic, with $(\kappa, \tau)\mapsto(\kappa, -\tau)$ and  $(a,b)\mapsto (a,-b)$ in equations \eqref{eq:geods1}-\eqref{eq:geods3}. 
\end{lemma}
\begin{proof}The transformation  of $\kappa,\tau$ follows from the Frenet-Serret equations \eqref{eq:FS}. The transformation of $(a,b)$ then follows from  equations  \eqref{eq:geods1}-\eqref{eq:geods3}. 
\end{proof}

Therefore in what follows we will  consider only the parameter values $a,b\geq 0$. See Figure \ref{fig:phase}.
\begin{figure}[h]\centering
\includegraphics[width=0.7\textwidth]{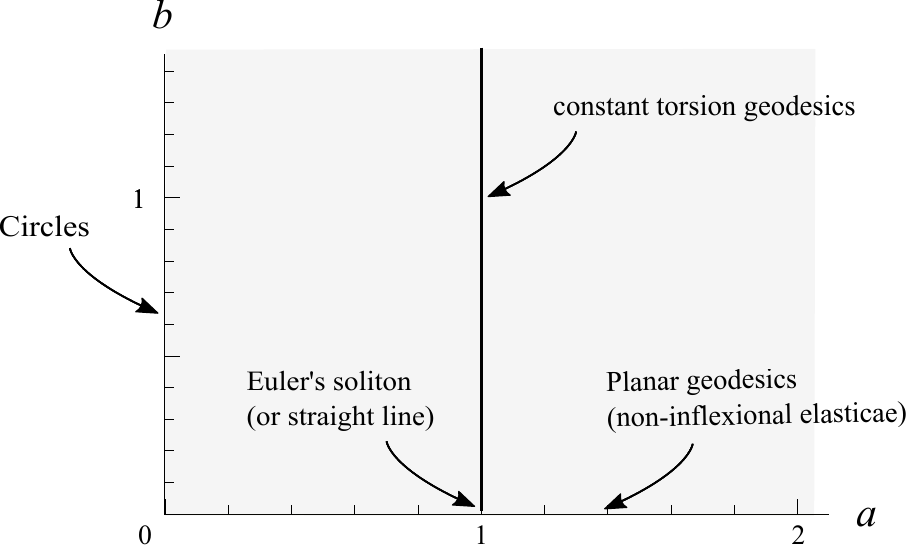}
\caption{The parameter space of bicycle  geodesics. } \label{fig:phase}
\end{figure}

\begin{lemma}
All values of the parameters $a,b\geq 0 $ in equations \eqref{eq:geods1}-\eqref{eq:geods3} occur among non-linear front tracks of bicycling geodesics. 
\end{lemma}

\bp Let $\v_0:=(1,0,0)$, $\x'_0:=(0,1,0)$, $\p:=(0,a,b)$. Then $\v_0\cdot (\x_0'-\p)=0$, so these are admissible initial conditions for equations \eqref{eq:mag_geod}. The solution is a bicycling geodesic with  $\p\cdot(\v_0\times \x_0')=b,$ $|\p|^2=a^2+b^2,$ as needed. 
\ep

\begin{prop}\label{prop:range}The curvature and torsion of  non-linear  geodesic front tracks, except for the Euler soliton ($a=1, b=0$),   are periodic elliptic functions, varying in the following ranges (note the `doubling discontinuity' of  the range of $\kappa$ at $a=1$):
\renewcommand{\arraystretch}{1.5} 
\begin{flalign*}
&\begin{array}{  r ll l}
{\rm (i)}& \mbox{For  } 0\leq  a < 1:  & \quad - \frac{ab}{1 + a}\ \leq \ \tau\ \leq\  \frac{ab}{1 - a}, & \qquad 1-a\leq \kappa\leq 1+a.\\
 {\rm (ii)}& \mbox{For  } a>1: & \quad \frac{ab}{1-a}\ \leq \ \tau\ \leq\   -\frac{ab}{1+a},&\qquad a-1\leq \kappa\leq a+1. \\
 {\rm (iii)}&\mbox{For  } a = 1:&\quad \tau=b/2,  &\qquad  -2\leq \kappa \leq 2.
\end{array}&&
\end{flalign*}
\end{prop}

\bp We use  equation \eqref{eq:geods2} to eliminate $\tau$ from   equation 
\eqref{eq:geods3}. Then, setting $u:=\kappa^2$, we get  
\be\label{eq:kbounds}
(u')^2 = P(u), 
\ee
where 
$$
P(u)=(u+ b^2)\left[  (1 + a)^2-u \right] \left[u - (1 - a)^2\right].
$$
Equation  \eqref{eq:kbounds} defines an oval in the  $(u,u')$ right half plane $u\geq 0$, an integral curve of the  vector field $u'\partial_u+{1\over 2} {dP\over du}\partial_{u'}$. This vector filed does not vanish along the oval, since $P$ has no multiple roots for $a,b>0$ (except for $a=1, b=0$,   the Euler soliton, see below). See  Figure \ref{fig:uphase}. 
\begin{figure}[h]
\center
\includegraphics[width=0.5\textwidth]{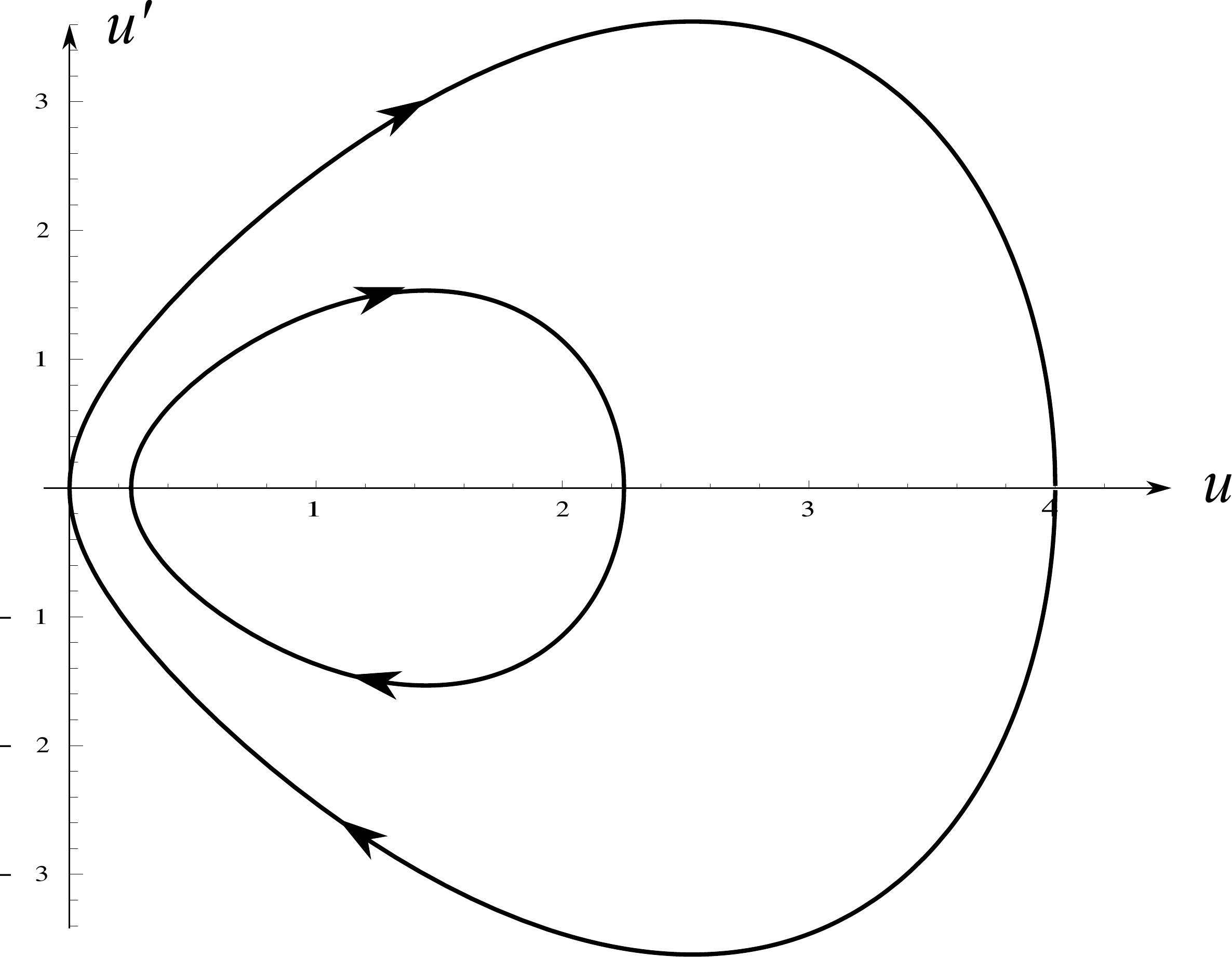}
\caption{Phase portrait of equation \eqref{eq:kbounds}, for $a=.5, b=1$ (inner oval), $a=1, b=1$ (outer  oval).}\label{fig:uphase}
\end{figure}

Consequently, the  phase point of equation \eqref{eq:kbounds}
moves  clockwise along this closed oval, so that $u$ oscillates periodically between the two non-negative  roots $(1\pm a)^2$ of $P(u)$. 
For $a\neq 1$, since $u=\kappa^2\geq (1-a)^2>0$, this gives the claimed range of $\kappa$.  

  For $a=1$ one needs to be more careful. See the outer oval of Figure \ref{fig:uphase}. The range of $u=\kappa^2$ is $[0,4]$, so $\kappa$ traverses the range $[0,2]$ as the phase point $(u,u')$ goes once around the oval, starting and ending at $u=u'=0$.  Now $u=0$ implies   $\kappa=0$, an inflection point of the front track, see Remark \ref{rmrk:FS}.  By equation \eqref{eq:geods3}, at this point $(\kappa' )^2=b^2>0$, so $\kappa$ changes sign. Going once more around the oval, $\kappa$ now traverses the range  $[-2,0]$.

The range of $\tau$ is  obtained  from that  of $\kappa$ via equation \eqref{eq:geods2}. For example, when $a< 1$, we have $\kappa^2(2\tau - b) = b(a^2-1) < 0$, so that
$2\tau - b < 0$ and
$$ (1 + a)^2(2\tau + b) \le b(1 - a^2) \le (1 - a)^2(2\tau + b)
$$
since $(1 - a)^2\le \kappa^2\le (1 + a)^2$. Rearranging, we find $\frac{ab}{1 - a} \le \tau \le -\frac{ab}{1+a}.$ The cases $a>1$ or $a = 1$ are similar. 


Since $P(u)$ is cubic in $u$, the solutions of \eqref{eq:kbounds} are     elliptic functions (doubly periodic in the complex  domain) 
and so are $\kappa(t)$ and $\tau(t)$.  See the Appendix for explicit formulas. 
\ep 

\begin{cor} \label{cor:range}
\begin{enumerate}[{\rm (a)}]

\item      Non-linear front tracks with  constant curvature  are  unit circles ($\kappa=1$, $a=0$).

\item     Non-linear front tracks with   constant torsion $\tau\neq 0$ correspond to $a=1$, $b=2\tau$.

\item   Front tracks  with $a>1$ have nowhere vanishing torsion, while those with $0 < a < 1$ have torsion of mixed signs. In both cases the curvature is non-vanishing (positive). 

\item    Front tracks with $a=1, b>0 $ (constant non-vanishing torsion) have curvature of mixed sign.

\item     Geodesic front tracks which are elastic curves ($a_2=0$ in equations \eqref{eq:rods1}-\eqref{eq:rods2}) are  planar  non-inflectional elasticae ($b=0$), as in Proposition \ref{prop:planar}(a) below. See Figure \ref{fig:elasticae}. 

\end{enumerate}
\end{cor}

All these statements follow immediately from Proposition \ref{prop:range}.

\begin{prop} \label{prop:planar}
\begin{enumerate}[{\rm (a)}]
 \item  Non-linear planar front tracks ($\tau=0$) are   non-inflectional elasticae ($b=0$), as in \cite{2D}. The  parameter values  of the Euler soliton coincide with those of  the straight line   ($a=1$, $b = 0$, see Lemma {\rm\ref{lemma:bc}(iii)}). The plane of the motion is parallel to $\p$. 
\item   The curvature of  non-planar front tracks with constant torsion ($a=1, b>0$)   is   that of  {\em inflectional} planar elasticae. 
\end{enumerate}
\end{prop}

\bp \begin{enumerate}[(a)]
\item By Proposition  \ref{prop:range},  $\tau=0$ occurs if and only if  $a=0$ or $b=0$, and  $a=0$ corresponds to unit circles.  If $b=0$ then equation \eqref{eq:geods3} becomes
$$(\kappa ')^2+\frac{1}{4} \left( 1 + a^2 - \kappa^2 \right)^2=a^2, 
$$
which is the equation for  non-inflectional elasticae appearing as planar geodesic front track, see \cite[Proposition 4.3]{2D}. 
Among these, the Euler soliton corresponds to $a=1$. The statement about the plane of the motion follows from formula 
\eqref{eq:pconst}.

\item If $\tau$ is constant and non-vanishing then, by  Proposition  \ref{prop:range},  $b=2\tau$ and $a=1.$ Equations \eqref{eq:geods1} and \eqref{eq:geods3} then become 
\be\label{eq:inf}
 \kappa'' +{1\over 2}\kappa^3+A\kappa=0, \quad (\kappa')^2 + \left( {\kappa^2\over 2}+A\right)^2=A^2+b^2, 
\ee
where $$ A={b^2\over 4}-1.$$ For $b> 0$ these are  equations for the curvature of  inflectional planar elasticae. See for example \cite[\S3.1]{2D}.
\end{enumerate}
\ep



\begin{prop} \label{prop:closed} The only closed bicycling geodesics are those whose front tracks are unit circles ($a=0$). 
\end{prop}

\bp By the second equation of  \eqref{eq:ham}, 
\be\label{eq:pv}
(\p\cdot\v)'= \p\cdot [\x' -  (\p\cdot \v)\v]=(\p\cdot \x)'-(\p\cdot \v)^2.
\ee
 Integrating this over a period, $ \int(\p\cdot \v)^2=0,$ so $\p\cdot \v=0.$ 
 
 It also follows from 
 \eqref{eq:pv} that $\p\cdot \x'=(\p\cdot \v)'+(\p\cdot\v)^2=0,$ and $\x'\cdot\v=(\p+\r)\cdot\v=\p\cdot \v =0.$ Thus $\p, \v, \x'$ are pairwise orthogonal. It  follows  that  
 $$
 b^2=|\p\cdot(\v\times\x')|^2=|\p|^2|\v|^2 |\x|^2=|\p|^2=a^2+b^2,
 $$ hence $a=0$,  and by Corollary  \ref{cor:range}(a) we have a unit circle. 
 \ep

\subsection{Period doubling}

Let us consider a geodesic front track with parameter values $(a,b)\neq (1,0)$ (all cases,  except a straight line and Euler's soliton). 

Denote the period of   $\kappa^2$ by $T.$ 
Using equation \eqref{eq:kbounds}, one can write it explicitly:
$$T=2\int _{(1-a)^2}^{(1+a)^2}{du\over \sqrt{(u+ b^2)\left[  (1 + a)^2-u \right] \left[u - (1 - a)^2\right]}}.
$$
(In the Appendix we express this integral using standard elliptic integrals.) Clearly, $T(a,b)$ is   continuous in $a,b>0$ (it is even analytic). 

For $a\neq 1$, one has $\kappa^2\geq (1-a)^2>0$, hence $T(a,b)$   is  also the period of $\kappa>0.$ However,  for $a=1$, as mentioned during the proof of Proposition \ref{prop:range}, there is a point along the front track  with $\kappa^2=0$,    an inflection point, where $\x''=0$. See the outer oval of Figure \ref{fig:uphase}.

This is  exactly the case mentioned in Remark \ref{rmrk:FS}. Furthermore, equation \eqref{eq:geods3} implies that at this point $(\kappa')^2=b^2>0,$ so $\kappa$ changes sign as $\x$ crosses this inflection point. It is not until $\x$ reaches the next inflection point, that $\kappa$ completes a full period. Thus,   at $a=1$ there is a {\em period doubling} phenomenon of the front track's curvature. See Figure \ref{fig:double}.  

 \begin{figure}[h]\centering

\includegraphics[width=0.5\textwidth]{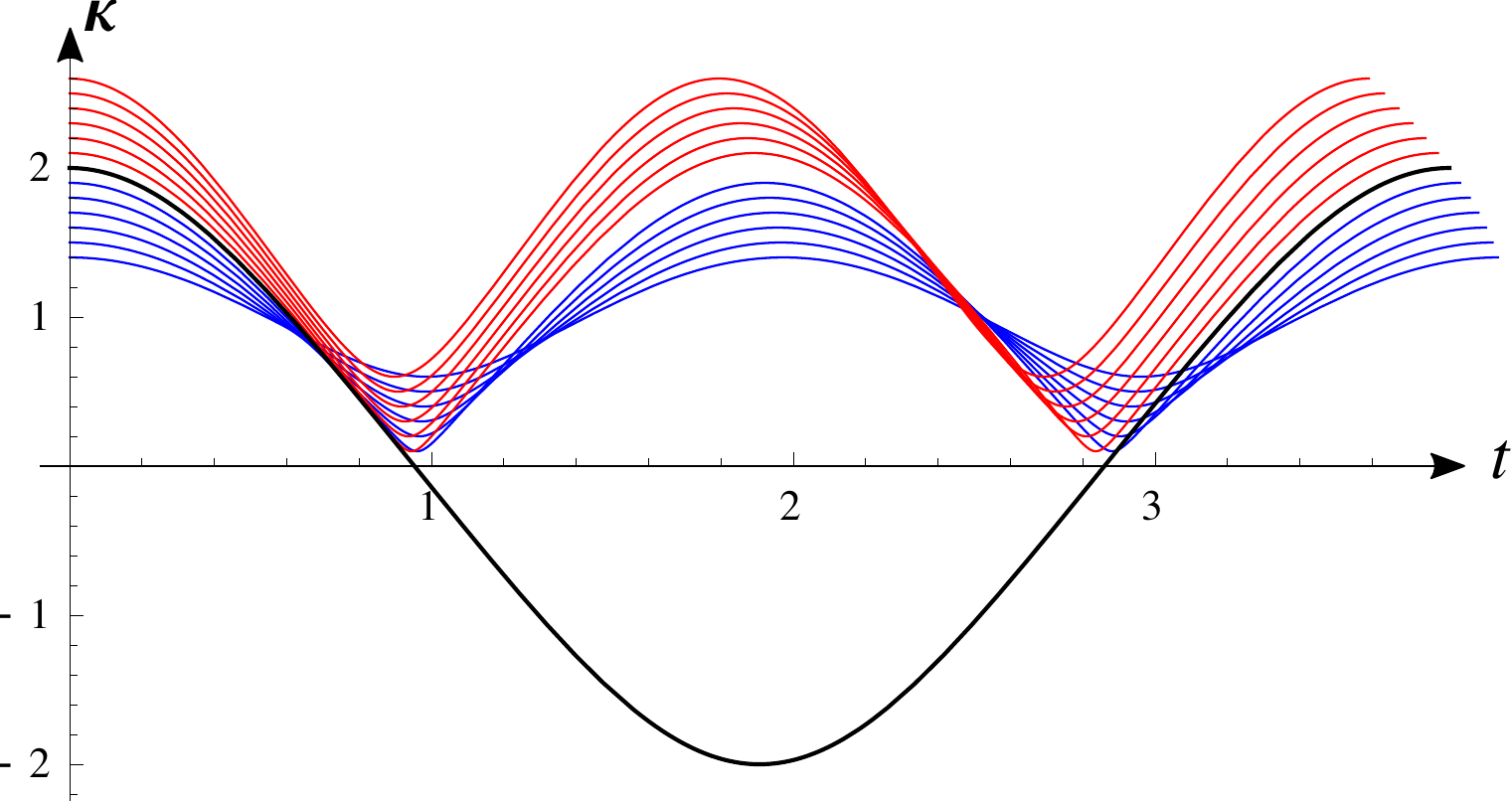}\quad 

\caption{Period doubling of the curvature of geodesic front tracks at parameter value $a=1$ (constant torsion). The figure shows a plot of the curvature $\kappa(t)$ of the front track over 2 periods of $\kappa^2$, for various values of $a$, at  fixed $b=1$. 
Blue: $a<1$. Red: $a>1$. Black: $a=1$.} \label{fig:double}
\end{figure}

\subsection{Back tracks}

We have focused so far on describing the front tracks  of  bicycling geodesics. In general, given a front track (a curve in $\R^3$), there is an  $S^2$-worth of associated  back tracks satisfying the no-skid condition, given by the initial frame position at some point along the front track.  For a linear front track, any back track   (a tractrix) will complete it to a bicycling geodesic. 

But this is an exception.  The next proposition states that back  tracks of all other bicycling geodesics are determined uniquely by their front tracks.

\begin{prop}\label{prop:bt}
Consider a non-linear  geodesic   front track $\x(t)$, parametrized by arc length. Then at a point of the front track with maximum curvature value,  where $\kappa =1+a $ (see Proposition \ref{prop:range}), the bicycle frame $\v$ is perpendicular to the front track and anti-aligned with  the   acceleration vector:  
%
%
\be\label{eq:ddx}\v=-{\x''\over 1+a}.
\ee
\end{prop}

\begin{proof}
Let  $F:=\v\cdot  \x',\ G:=\p\cdot \x'$. Using  equations \eqref{eq:ham} and $\r\cdot \v=0,$ 
one has
\be\label{eq:G}
G' = \p\cdot\r'=\p\cdot\left[(\v\cdot\p)\r-(\r\cdot\x')\v\right]=F \left( 2 G - |\p|^2 - 1\right).
\ee
Dotting equation \eqref{eq:pconst} with $\T=\x'$, we get  
\be\label{eq:2G}
2G = 1 + a^2 - \kappa^2,
\ee
whose derivative is 
\be\label{eq:Gp} G' = - \kappa\kappa'.
\ee
 We now calculate at a maximum point  of $\kappa$, where   $\kappa=1+a$ and $\kappa' = 0$. By equations  \eqref{eq:2G}-\eqref{eq:Gp}, 
  $G =-a$ and  $G'=0.$ By equation  \eqref{eq:G}, 
  $$0=F[-2a-(a^2+b^2)-1]=-F[(1+a)^2+b^2],$$ hence  $ F = \v\cdot\x'=0.$
Then by equations \eqref{eq:ham}, 
$$\v\cdot\p=\v\cdot(\x'-\r)=\v\cdot \x'-\v\cdot\r=0$$ and 
$$\r\cdot\x'=(\x'-\p)\cdot\x'=1-G=1+a,$$ so 
$$\x'' = \r' = (G- 1) \v=-(1+a)\v,$$ as needed. 
\end{proof}

\begin{rmrk}A  similar argument shows that $\v$ is alligned ($a > 1$) or anti-aligned ($a \le 1$ ) with  $\x''$ also at points of minimum curvature. These are the  points with $\kappa=|1-a|$, where $(a - 1)\v = \x''$, for $a\neq 1$ or $\kappa=-2$ for $a=1$, where $2\v = - \x''$. 

Another notable case is that of  an inflection point, where  $\kappa=0$, occurring  for $a=1$ half way between adjacent maxima and minima of $\kappa$. See Remark \ref{rmrk:FS}.
At such a point,  $\x''=0$ but $\x'''\neq 0$. The Frenet-Serret frame then extends  analytically to the inflection point via $\N=\pm \x'''/|\x'''|$, $\B=\T\times \N$, and  $\v$ is  aligned with $\pm\B$. \end{rmrk}

\subsection{Rescaling bicycling geodesics, with a torsion shift}

Kirchoff rods (solutions to equations \eqref{eq:rods1}-\eqref{eq:energy}) comprise -- up to  isometries -- a 4-parameter family of curves. The family is invariant under rescaling: if a Kirchoff rod $\x(t)$,  parametrized by arc-length, is scaled to $\tilde \x(t) := \lambda \x(t/\lambda)$, then the curvature and torsion scale by
 $$
 \tilde\kappa(t) = \frac{1}{\lambda}\kappa\left(\frac{t}{\lambda}\right),\ \quad \tilde\tau(t) = \frac{1}{\lambda}\tau\left(\frac{t}{\lambda}\right). 
 $$
From these formulas one can see  that $\tilde\x(t)$ is still a Kirchoff rod, satisfying equations \eqref{eq:rods1}-\eqref{eq:energy} with parameters:
\be\label{eq:scale}
 \tilde a_1 = \frac{a_1}{\lambda^2} ,\quad\tilde a_2 = \frac{a_2}{\lambda} ,\quad\tilde a_3 = \frac{a_3}{\lambda^3}, \quad\tilde a_4 = \frac{a_4}{\lambda^2}.
\ee
 
 So -- up to  similarities -- the Kirchoff rods define  a 3-parameter family of curves, i.e., a 3-parameter family of  {\em shapes}. The front tracks of  bicycling geodesics form a 2-parameter subfamily of Kirchoff rods (Theorem \ref{thm:main}(b)).

We consider how a bicycling geodesic (of a fixed frame length) might be rescaled by $\lambda>0$, $\lambda\neq 1$ (rescaling by $\lambda=-1$ is  realized by $(a,b)\mapsto (a,-b)$, see Lemma \ref{prop:ref}). 

We know that the planar front tracks ($b=0$), apart from the circle ($a=0$), line, and Euler's soliton ($a=1,\ b=0$), come in two  `sizes', `wide' and `narrow',   related by $a\mapsto 1/a$,  with the scaling factor $\lambda=a$, see \cite[\S4]{2D}.  For the spatial geodesics this is not the case. 

\begin{prop}
 A non-planar bicycling geodesic front track (with a fixed frame length) may not be rescaled.
\end{prop}
 \begin{proof} 
By equations \eqref{eq:ab}, the \Kh parameters of  a geodesic front track  satisfy 
$a_3=2a_2(a_1-1).$ 
After rescaling  by $\lambda$ this   becomes  $\lambda^{3}\tilde a_3 = 2\lambda \tilde a_2 ( \lambda^2\tilde a_1 - 1),$ or $\tilde a_3 = 2 \tilde a_2 ( \tilde a_1 - \lambda^{-2})$. For non-planar geodesics, $a_2 = b$ and $\tilde a_2$ are non-zero, so that only for $\lambda^{-2} = 1$ is the rescaled front track as well a bicycling geodesic.
\end{proof}

Nevertheless, the involution $a\mapsto 1/a$ for planar geodesics ($b=0$) can be extended to non-planar geodesics $b>0$, provided   one acts on space curves, in addition to rescaling, by a `torsion shift'. (We are indebted to David Singer for suggesting  this idea.) Here are the details. 


Let  $\kappa(t), \tau(t)$ be the curvature and torsion functions of a non-circular geodesic front track $\x(t)$ 
in $\R^3$, parametrized by arc length, satisfying equations  \eqref{eq:geods1}-\eqref{eq:geods3} with  parameter values $a,b\in\R$, $a> 0$.  Let us consider the new parameters values   
$$\tilde a:={1\over a}, \quad \tilde b:={b\over a},
$$
 and let $\tilde\kappa, \tilde\tau$  be the  curvature and torsion functions of the  associated new front track. 
\begin{prop}\label{prop:inv}  One has 
$$\tilde\kappa(t)={1\over a}\kappa\left({t\over a}\right),\quad 
\tilde\tau(t)={1\over a}\tau\left({t\over a}\right)-{b\over a}.
$$
 Namely, the new front track  is obtained from the old one  by torsion shifting,  $\tau\mapsto \tau-b$,  followed by rescaling by $\lambda=a$. 
 \end{prop}
The proof is by direct substitution in equations \eqref{eq:geods1}-\eqref{eq:geods3}.

\subsection{Monodromy}

 Consider a  geodesic     front track $\x(t)$, parametrized by arc length, with a periodic curvature function $\kappa(t)$. By Proposition \ref{prop:range}, this occurs for all geodesic front tracks,  except lines, circles, and Euler's solitons. Furthermore, denoting the period of $\kappa^2$ by $T$,  $\kappa$ is $T$-periodic for $a\neq 1$ and $T$-antiperiodic for $a=1$:
$$\kappa(t+T)=\kappa(t) \mbox{ for } a\neq 1, \quad \kappa(t+T)=-\kappa(t) \mbox{ for } a=1.
$$
By equation \eqref{eq:geods2}, the torsion   is then either $T$-periodic for $a\neq 1$, or constant for $a=1$. In both cases,  
$$ \tau(t + T) = \tau(t), \ \mbox{for all $t$.} 
$$
The  next proposition follows from the above (anti-)periodicity of $\kappa, \tau$ and the `fundamental theorem of space curves', except for a small twist in the antiperiodic case.   
\begin{prop}[and definition of Monodromy] \label{prop:mono} Given a  bicycle geodesic $(\x(t), \y(t))$ 
whose front track's curvature is $T$-periodic or antiperiodic, 
there is a unique  proper rigid motion  $M:\R^3\to\R^3$ (an orientation preserving  isometry), called the {\em monodromy} of $\x(t)$, such that 
\[   \x(t + T) = M(\x(t)),\ \y(t + T) = M(\y(t)), \ \mbox{for all $t\in\R$.} \]
\end{prop}

\bp Let $\tilde\x(t)=\x(t+T)$. If $a\neq 1$ then $\x(t), \tilde\x(t)$ have no inflection points, with the  same  curvature and  torsion functions. By the `fundamental theorem of space curves' \cite[\S21]{G} (we review it in the next paragraph), there is an orientation preserving isometry $M$ such that $\tilde\x(t)=M(\x(t))$ for all $t$. 

The uniqueness   follows from the non-linearity of $\x(t)$. The non-linearity of $\x(t)$  also implies that   $\y(t)$  is determined by $\x(t)$  (Proposition \ref{prop:bt}), hence $\x(t + T) = M(\x(t))$ implies that $\y(t + T) = M(\y(t)). $

For $a= 1$, when $\kappa$ is $T$-antiperiodic, we need to generalize  slightly  the `fundamental theorem of space curves'. Let us revise first the standard statement and proof of this theorem. 

One is given two curves in $\R^3$, $\x(t)$ and $\tilde\x(t)$, 
both parametrized by arc length, without inflection points (i.e., with non-vanishing acceleration),  
with Frenet-Serret frames satisfying the Frenet-Serret equations \eqref{eq:FS}, 
with the same curvature and torsion functions. The statement is then  that there is an orientation preserving isometry $I:\R^3\to\R^3$ such that $\tilde\x(t)=I\x(t)$ for all $t$. 

To prove it, one takes the isometry $I$ that maps $\x(0)$ to $\tilde\x(0)$ and the Frenet-Serret frame of the first curve at 
$t=0$ to that of the second curve at $t=0$. Then, since the Frenet-Serret equations are invariant under isometries, one gets, by the uniqueness theorem  of solutions to ODEs, that $I$ must take the whole Frenet-Serret frame of the first curve to that of the second. In particular, $I_*\x'(t)=\tilde\x'(t)$, which implies  $I\x(t)=\tilde\x(t)$ since $I\x(0)=\tilde\x(0)$. 

Now we observe that in this argument neither  the uniqueness of the Frenet-Serret frame was used, nor any assumption about the  curvature function (except smoothness). We can thus apply it in our case of $a=1$ by fixing a Frenet-Serret frame $\T(t), \N(t), \B(t)$ along $\x(t)$, with the corresponding curvature and torsion functions $\kappa(t), \tau(t)$ (there are two choices of the frame, we pick one of them). Along $\tilde\x(t)=\x(t+T)$, we pick  the other  choice: 
$$\tilde\T(t)=\T(t+T), \ \tilde\N(t)=-\N(t+T), \ \tilde\B(t)=-\B(t+T).$$
  Then we can check that  this frame satisfies the Frenet-Serret equations with curvature and torsion functions $\tilde\kappa(t)=-\kappa(t+T)=\kappa(t), \tilde\tau(t)=\tau(t+T)=\tau(t)$, so there is an isometry $M$ mapping $\x(t)$ to $\tilde\x(t)$, as needed. 
\ep

Next recall that every proper rigid motion  in $\R^3$ is a `screw motion',  the composition of translation and rotation about a line, the rotation axis of the motion  (the  {\em Chasles Theorem}). We shall now find the rotation axis of the monodromy of Proposition \ref{prop:mono}.

By Theorem \ref{thm:main}(c), a geodesic front   track  is the trajectory  of a charged particle in a magnetic field 
$\K$, a Killing field  generating a screw motion  about the line passing  through $\x_1$ and parallel to $\p$, the  rotation axis of $\K$. 

If the front track is planar then, by Proposition \ref{prop:bt}, the geodesic is planar, and this is the case studied in \cite{2D}. Therefore we consider non-planar front tracks in what follows.

\begin{prop}\label{prop:screw}
The monodromy $M$  of a non-planar geodesic front track with periodic curvature is a screw motion with  axis  parallel to  the axis of the associated  magnetic field  $\K$. If the rotation part of $M$ is non-trivial then its axis coincides with that of  $\K$. 


\end{prop}

\begin{proof}
%
We first show that the translation part of $M$ is non-trivial. To this end, we will show that $\p\cdot \x(t)$ is unbounded. 
 
 Using  equations \eqref{eq:G}-\eqref{eq:2G}, one has 
$$ \p\cdot\x'=\frac{1+a^2-\kappa^2}{2}, \quad  \p\cdot\v=-{\p\cdot\x''\over \kappa^2+b^2}.
$$
It follows that  $\p\cdot\x'$ and $ \p\cdot\v$ are $T$-periodic. 
Next,  integrating   equation \eqref{eq:pv},
$$(\p\cdot \x)'=(\p\cdot\v)'+(\p\cdot \v)^2,
$$
over $[0,T]$, and using the periodicity of $\p\cdot\v$, we get that 
$$\int_0^{T}\p\cdot \x'=\int_0^T(\p\cdot\v)^2>0,
$$
 unless $\p\cdot\v=0$ over the whole period. This would imply, by equation \eqref{eq:Gp}, that $\kappa'=0$, i.e., the front track is a circle, which has been excluded.  It follows that 
 $$\p\cdot\x(nT)=\p\cdot\x_0+n\int_0^{T}\p\cdot \x'$$
  is unbounded,  hence  the translation part of $M$ is non-trivial, along an axis parallel to the axis of $\K$. If the rotational part of $M$ is trivial, then $M$ is a pure translation along this axis of $\K$. 
  
    Assume now that the rotational part of $M$ is non-trivial. Let $r(t)$ be the distance from $\x(t)$ to the rotation axis of $\K$. From equations \eqref{eq:rel} and \eqref{eq:mag_geod} one has 
 $$-b\T-\kappa\B=(\x-\x_1)\times \p+\delta\p,
 $$
  and taking the square norm of both sides gives 
 \be\label{eq:r}b^2+\kappa^2=r^2|\p|^2+\frac{b^2}{|\p|^2}.
 \ee
Hence $r(t)$ is $T$-periodic. (In fact, this is a general property of  Killing magnetic field  trajectories, see \cite[equation (3.2)]{Mo}.) 
  

Let us project the front track $\x(t)$  onto the plane orthogonal to the axis of $\K$, and assume that this axis projects to the origin $O$. We obtain a planar curve $\bar \x(t)$,  invariant under a rotation $\bar M$,  and we need to show that the center of this rotation, say $F$, is the origin.

If the rotation angle is not equal to $\pi$ then pick a point $\u=\bar\x(t_0)$, not equal to $F$ (such a point exists else $\x(t)$ is a linear track), and consider the three non-collinear  points $\u, \bar M(\u), \bar M^2(\u)$. These points are at equal distances from $F$ and, by the $T$-periodicity of $r(t)$ (see equation \eqref{eq:r}), at equal distances from $O$. Hence $F=O$ is the circumcenter of the triangle with vertices $\u, \bar M(\u), \bar M^2(\u)$.
  
  If the rotation angle equals $\pi$ then  for each point $ \bar\x(t)$ the points $ \bar\x(t)$ and $\bar M( \bar\x(t))$ are symmetric with respect to $F$ and are at equal distances from $O$. 
  If $F\neq O$ then it follows that  the whole curve $\bar\x(t)$ lies on the line orthogonal to $FO$ and passing through   $F$, hence $\x(t)$ is planar, contradicting our original assumption (the planar geodesics have already been described in \cite{2D}, having purely translational monodromy).  
  \end{proof}



\subsection{Bicycle correspondence}
 The bicycling configuration space $Q=\R^n\times S^{n-1}$ is equipped with a \sR\ structure whose geodesics are the bicycling geodesics considered in this article. 

Identifying $Q$ with the tangent unit sphere bundle on $\R^n$, the   Euclidean group acts naturally on $Q$, preserving this \sR\ structure, hence it  acts also on the space of \sR\ geodesics on $Q$. Theorem \ref{thm:main} describes the geodesics up to this action. 

Now there is an additional \sR\ isometry,   $\Phi:Q\to Q$, an involution, not coming from the said Euclidean group action, called {\em bicycling correspondence}   (a.k.a. the Darboux-B\"acklund transformation of the filament  equations \cite{RS,Ta}). It is defined 
by  `flipping the bike about its back wheel': 
$$\Phi: (\x,\v)\mapsto (\x-2\v, -\v).$$
Thus, when acting by $\Phi$ on a bicycle path, the back track is unchanged, while the front track is `flipped'. See Figure \ref{fig:flip}. 

 \begin{figure}[h]\centering
\def\svgwidth{.35\textwidth}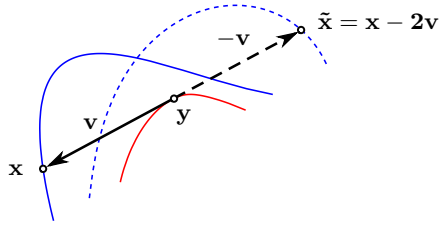
\caption{Bicycle correspondence. The back track (red) is unchanged and the front tracked is `flipped' (from solid blue to dashed blue). } \label{fig:flip}
\end{figure}

One can verify  that $\Phi$  is a \sR\ isometry, i.e., it preserves the horizontal distribution  $D\subset TQ$ and the  \sR\ metric on it (see Lemma \ref{lemma:sR} below), hence it acts on the space of bicycle geodesics. 

For $n=2$ this action was studied in \cite{2D} (see Proposition 4.11 and Figure 8). It was found that, with one notable exception, $\Phi$ acts on the front tracks of geodesics by translations and reflections, i.e., by Euclidean isometries. The notable exception is a bicycle geodesic with linear front track and non-linear back track (a tractrix), which $\Phi$ transforms into a bicycle geodesic whose front track is the Euler soliton (and vice versa). 

Here we study  this action for $n=3$. What we find is that, with the same exception as for $n=2$, the bicycle correspondence transforms the front tracks of bicycle geodesics by a rigid motion $I\in{\rm Iso}(\R^3)$, a `square root of the monodromy':   $I^2=M. $ 


To begin with, let us verify the claim made above.

\begin{lemma}\label{lemma:sR} 
$\Phi:Q\to Q$ is a \sR\ isometry. 
\end{lemma}

\begin{proof}  We first show that $\D$ is $\Phi$-invariant. Let $(\x,\v)\in Q$ and $ (\x',\v')\in D_{(\x,\v)}.$  
That is,    
$$|\v|=1, \ \v'=\x'-(\x'\cdot \v)\v.
$$
Then 
$\Phi(\x,\v)=(\tilde\x, \tilde\v),\ \Phi_*(\x',\v')=(\tilde\x', \tilde\v'),$
where 
 $$\tilde\x=\x-2\v, \ \tilde\v=-\v,\ \tilde\x'=\x'-2\v', \ \tilde\v'=-\v'.
 $$ 
One has then 
\begin{align*}
\tilde\v'-\left[ \tilde\x'-(\tilde \x'\cdot \tilde\v)\tilde\v\right]
&=-\v'-(\x'-2\v')+\left[(\x'-2\v')\cdot\v\right]\v\\
&=\v'-\left[\x'-(\x'\cdot\v)\v\right]=0,
\end{align*}
hence  $(\tilde\x', \tilde\v')\in D_{(\tilde\x, \tilde\v)}$. 

Next 
\begin{align*}
|\tilde\x'|^2&=|\x'-2\v'|^2=|\x'|^2+4\v'\cdot(\v'-\x')\\
&=|\x'|^2+4[\x'-(\x'\cdot\v)\v]\cdot[(\x'\cdot\v)\v]=|\x'|^2,
\end{align*}
hence $\Phi$ is a \sR\ isometry.
 \end{proof}

Now consider   a bicycle geodesic  $(\x(t),\v(t))$ in $Q$ whose front track's curvature is $T$-periodic or anti-periodic, and with monodromy $M$, as in Proposition  \ref{prop:mono}. 
Let 
$$(\tilde\x(t),\tilde\v(t))=\Phi(\x(t),\v(t)),\ {\rm i.e.,}\ \ 
\tilde\x(t)=\x(t)-2\v(t), \ \tilde\v(t)=-\v(t).$$ 

We assume that $\x$ is not planar. 

\begin{prop}\label{prop:bc}There is a screw motion  $I$, with the same  axis as  $M$, such that 
\begin{enumerate}[{\rm (a)}]
%
%
\item $\tilde\x(t+T/2)=I(\x(t)),\ \y(t+T/2)=I(\y(t))$ for all $t$. 
\item 
$I^2=M.$ 
\end{enumerate}
\end{prop}

\begin{proof} Similarly to the proof of Proposition \ref{prop:mono}, to show that $\tilde\x(t+T/2)$ and $\x(t)$ are related by a proper isometry it is enough  to show that  their curvature and torsion functions coincide: 
$$
\tilde\kappa(t+T/2)=  \kappa(t), \ \tilde\tau(t+T/2)=\tau(t) \mbox{ for all }t.
$$

\begin{lemma}\label{lemma:param} $\tilde\x(t)$ and $\x(t)$  have $\tilde\p = \p$ and the same parameter values, $ (\tilde a, \tilde b)=(a,b)$.
\end{lemma}

\bp  We will first show that $\tilde \p = \p$. From  $|\v|=1$ we have $\v\cdot\v'=0$, hence 
\be\label{eq:xpv}\tilde \x'\cdot\tilde \v=(\x'-2\v')\cdot(-\v)=-\x'\cdot\v+(\v\cdot\v)'=-\x'\cdot\v. 
\ee
Then, from the first equation of \eqref{eq:ham}, $\x'=\p+\r$,  and $\r\cdot\v=0$,  we have 
$$\p\cdot\v=(\x'-\r)\cdot\v=\x'\cdot\v, \mbox{ and similarly, }\tilde\p\cdot\tilde\v=\tilde\x'\cdot\tilde\v.
$$
 Thus, 
$$ \p\cdot \v = - \tilde \p \cdot \tilde \v =\tilde\p\cdot\v\so 0=(\tilde\p-\p)\cdot\v.
$$
Differentiating, 
$$
 \p\cdot \x' = \tilde \p \cdot \x'\so (\tilde\p-\p)\cdot\x'=0.
$$
For non-planar  $\x(t)$, the last equation  implies $\p = \tilde\p.$ 

Next, $\tilde\x=\x-2\v$ implies, by equation \eqref{eq:noskid}, 
$$\tilde\x' =\x'-2\v'= - \x' + 2(\x'\cdot\v)\v,$$ 
hence
$$ \tilde b = \det(\p, \tilde \v, \tilde \x') = \det( \p, - \v, -\x') = b.
$$

Finally,   $|\p|^2=a^2+b^2$  and  $|\tilde\p|^2=\tilde a^2+\tilde b^2$, hence from $\tilde\p=\p$ and $\tilde b=b$ it follows that $\tilde a=a$.
\ep



\begin{lemma}\label{lemma:cp}  The critical points of $\kappa^2$ are the points where $\x'\cdot\v=0$, and they are maxima or minima. 
Bicycle correspondence maps the critical points of $\kappa^2$ to those of $\tilde\kappa^2$, interchanging maxima and minima. 
\end{lemma}

\begin{proof}
From equations \eqref{eq:G}-\eqref{eq:Gp} we get
$$(\kappa^2)'=2\kappa\kappa'=-2G'=-2F(2G-|\p|^2-1)=2(\x'\cdot\v)(\kappa^2+b^2). 
$$
Now $\kappa^2+b^2$ is non-vanishing:  if it does then  $b =0$ and $\kappa $  vanishes, which cannot happen, since planar non-linear geodesic front tracks are non-inflectional elasticae (see Proposition \ref{prop:planar}). Thus critical points of $\kappa^2$ are points where $\x'\cdot\v$ vanishes. 

It follows from equation \eqref{eq:kbounds} that critical points of $\kappa^2$ are maxima or minima, where $\kappa^2=(1\pm a)^2.$  Now,  from equation \eqref{eq:xpv} we have  $\tilde \x'\cdot\tilde\v=-\x'\cdot\v,$  so $(\kappa^2)'$ and $(\tilde\kappa^2)'$ have opposite signs. The critical points of $\kappa^2$ are isolated, hence  the derivative changes sign at a critical point, from positive to negative at a maximum, and from negative to positive at a minimum. Similarly for  $\tilde\kappa^2$. Since the derivatives of these functions have opposite signs, it follows that when $\kappa^2$ is at a maximum   $\tilde\kappa^2$ is at a minimum, and vice-versa, as needed. 
 \end{proof}
 




\mn{\bf  Proof of  Proposition} \ref{prop:bc}(a). By Lemma \ref{lemma:param}, $\kappa$ and $\tilde\kappa$  have the same period (or anti-period for $a=1$), $T=\tilde T$, and differ at most by a parameter shift. By Lemma \ref{lemma:cp}, the parameter shift is $T/2$ (or any odd multiple of $T/2$). By equation \eqref{eq:geods2}, the curvature determines the torsion, hence $\tau, \tilde\tau$ are also related by the same  parameter shift. By the `fundamental theorem of space curves', there is a proper isometry mapping $\x(t)$ to $\tilde\x(t+T/2)$. 

The  statement $\y(t+T/2)=I(\y(t))$ now follows: $(I(\x(t)), I(\y(t)))$ and $(\x(t+T/2), \y(t+T/2))$ are  geodesics with the same non-linear front track; by Proposition \ref{prop:bt}, they  have the same back track.

\mn{\bf  Proof of  Proposition} \ref{prop:bc}(b). For any bicycle path $(\x(t), \y(t))$ we use the notation $\tilde\x(t):=B_{\y(t)}\x(t)=2\y(t)-\x(t)$ (`flipping of the front track $\x(t)$ with respect to the back track $\y(t)$'). This operation has the following obvious properties:
\begin{itemize}
\item $\x(t):=B_{\y(t)}\tilde\x(t)$.
\item For any isometry $I$,  $(I(\x(t)), I(\y(t)))$ is also a bicycle path and $I(\tilde\x(t))=B_{I(\y(t))}I(\x(t)).$
\item $\tilde\x(t+t_0)=B_{\y(t+t_0)}\x(t+t_0)$ for any $t_0\in\R$. 
\end{itemize}
Now applying these properties and the previous item, we calculate
\begin{align*}
I^2(\x(t))&=I(I(\x(t)))=I(\tilde\x(t+T/2))=I(B_{\y(t+T/2)}\x(t+T/2))\\
&=B_{I(\y(t+T/2))}I(\x(t+T/2))=B_{\y(t+T)}\tilde\x(t+T)\\
&=\x(t+T)=M\x(t).
\end{align*}
Thus $I^2=M$ since $\x(t)$ is non-linear.  
\ep

\subsubsection{A conjecture} Let $\Delta\theta\in[0,2\pi)$ and  $\Delta z>0$ be the rotation angle and  translation  of the monodromy  about its   axis. 
Proposition \ref{prop:bc} implies that the translation of $I$ is $\Delta z/2$, but it does not determine the rotation angle uniquely: it may be either $\Delta\theta/2$ or $\Delta\theta/2+\pi$.
Based on numerical evidence, and again assuming that the geodesic is not planar, we make the following conjecture.

\begin{conjecture}\label{conj:rot}
The rotation angle of $I$ is $\Delta\theta/2+\pi$,
see Figures \ref{fig:bc}.
\end{conjecture}
\begin{figure}[h]\centering
\includegraphics[width=.9\textwidth]{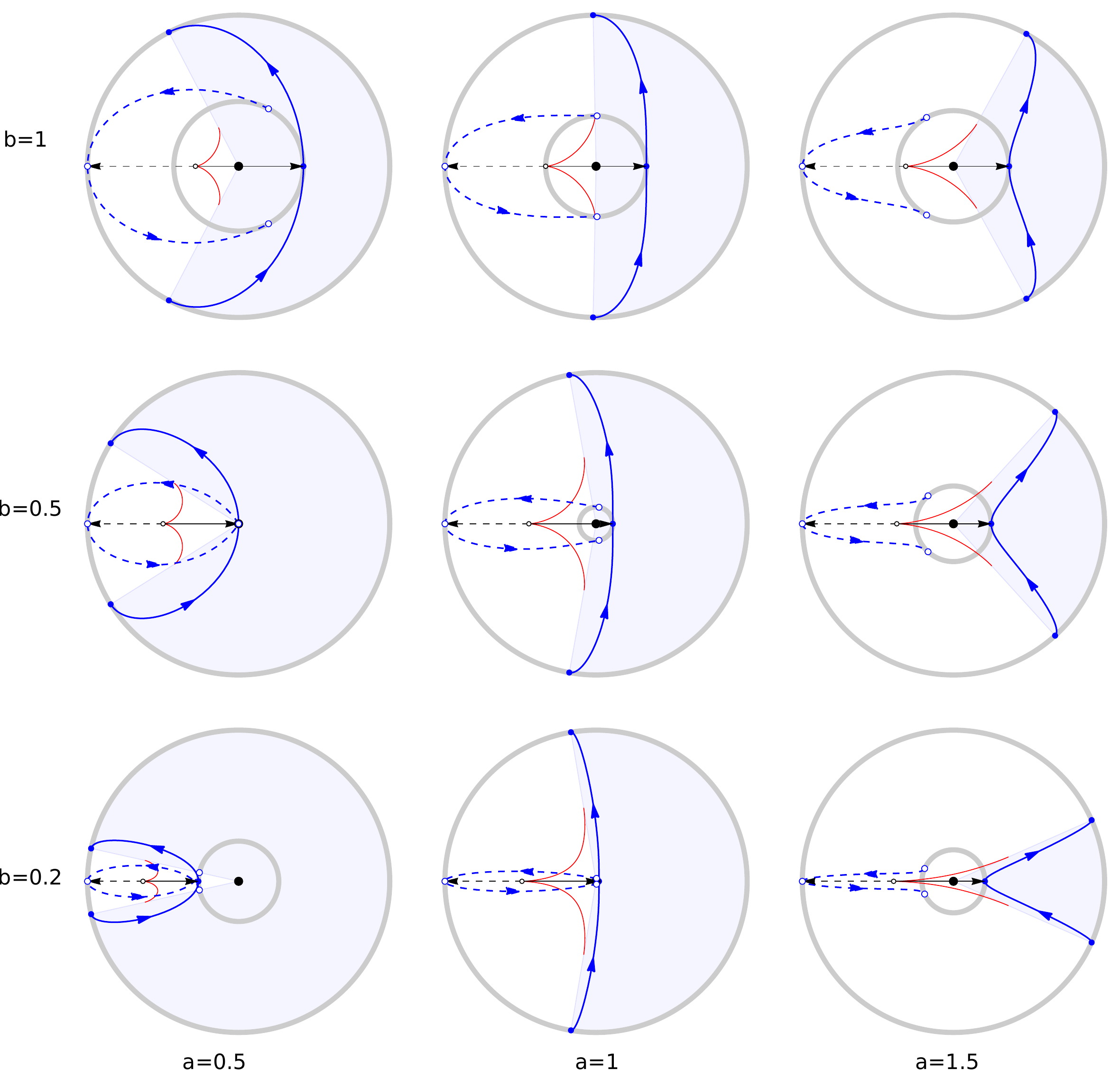}
\caption{Numerical evidence supporting Conjecture \ref{conj:rot}. Bicycle correspondence of  geodesics, projected onto $\p^\perp$, is shown for various values of the parameters $a,b$. The solid blue curve is $\x(t)$, the dashed curve is $\tilde\x(t)$, the red curve is their common back track $\y(t)$, all drawn in the range $0\leq t\leq T$, between two successive points of maximum  distance of $\x(t)$  to the rotation axis. The rotation angle  of the monodromy is  marked by the darkened sector. Also marked is the bicycle correspondence between $\x(T/2)$ and  $\tilde\x(T/2)$. 
} \label{fig:bc}
\end{figure}

\subsection{Global minimizers}
Bicycle geodesics, by definition,  have critical length among bicycle paths  connecting  two given  placements of the bike frame. In particular, some of them  are the {\em minimizing }bicycle paths. We shall not study them in detail but will only find the {\em global minimizers}, namely, the bike   paths $(\x(t), \y(t))$,   $t\in\R$, which are minimizers for any of their finite subsegments. 

There are two obvious candidates: those whose front tracks $\x(t)$ are straight lines or Euler's solitons. Are there any other ones? 

The answer is no. The reason is that, as we know, all other bicycle geodesics are {\em quasi-periodic}:  $\x(t+T)=M(\x(t))$,  $\y(t+T)=M(\y(t))$  for some $T>0$ and all $t$. Here is the detailed argument. 

\begin{prop}\label{prop:global}
A quasi-periodic non-linear bicycle path is not a global minimizer. 
\end{prop}
\begin{proof} 
The argument is  taken from  \cite[pages 4675-6]{2D}, whose Figure 11 is reproduced here with minor changes as Figure \ref{fig:shortcut}. 

 \begin{figure}[h]\centering
\includegraphics[width=0.7\textwidth]{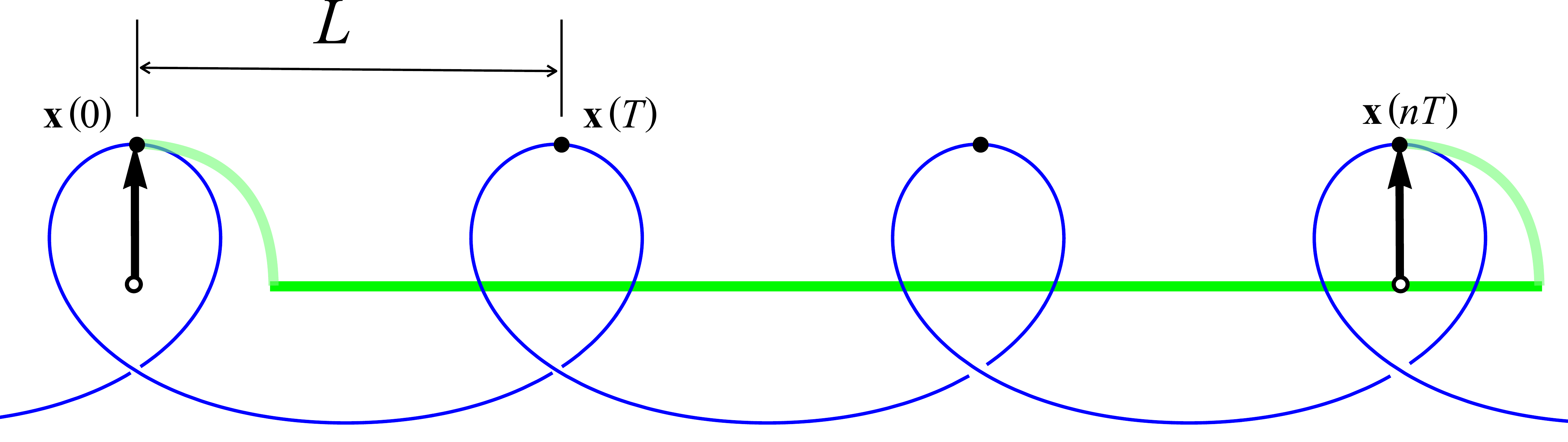}\quad 
\caption{A shortcut.} \label{fig:shortcut}
\end{figure}

Since $\x(t)$ is non-linear, there are two values of time, $T$ apart, say $0$ and $T$, so that the segment of $\x(t)$ between $\x(0)$ and $\x(T)$ is not a line segment. It follows that   $L:=|\x(T)-\x(0)|< T$.   

Now take the front track segment  with end points at $\x(0),\ \x(nT)$ for some positive integer $n$. Its length is $nT$ and the distance between end points is at most $nL$. 

Let the back track be $\y(t)$. Then $|\y(T)-\y(0)|\leq nL +2.$ So, for $n$ big enough, one can do better then $nT$ by
\begin{itemize}
\item reorienting the bike at $t=0$, with fixed back wheel, so it points to $\y(nT)$;
\item ride straight towards  $\y(nT)$;
\item reorient the bike so its front wheel is at $\x(nT).$ 
\end{itemize}
Steps 1 and 3 cost at most some fixed amount  independent of $n$, and step 2 costs at most  $nL+2$. So total cost is at most $nL + c$, for some $c$ independent of $n$. This is less then $nT$ for $n$ big enough. 
\end{proof}
\appendix

\section{Explicit formulas}

One can get explicit formulas for bicycling geodesics in $\R^3$ as a special case of those for \Kr, as  in \cite[\S4]{LS}, but we found it actually easier to obtain  them directly. (The  formulas  of \cite{LS} require  first solving  complicated algebraic equations for the parameters $p,w$ appearing  in those formulas.) We use mostly  the notation of \cite{RW}. 

\begin{prop}\label{prop:exp} 
\begin{enumerate}[(a)]
\item The curvature of a non-linear geodesic front track, parametrized by arc length, is   given by 
%
\be
\label{eq:sn}
\kappa^2(t)=(1+a)^2 -4a\,{\rm sn}^2\left(\omega t, k\right)
\ee
where 
\be\label{eq:params}\omega={\sqrt{ (a+1)^2+b^2}\over 2}, \qquad k^2 ={4a\over  (a+1)^2+b^2}. 
\ee
Here ${\rm sn}(u,k)$ is the Jacobi elliptic function with  modulus  $k$ (or parameter $m=k^2$). 

\item  For $a\ne 1$  the period of the curvature $\kappa$ is the same as the period of $\kappa^2$, given by
$$T = {2K(k)\over \omega},$$
where  $K(k)$  is the  complete elliptic integral of the first kind. 

\item For $a=1$ (constant torsion front tracks)  the front track's curvature and its period are 
$$
\kappa(t)=2{\rm cn}\left(\omega t,k\right), 
 \quad 2T = \frac{4K(k)}{\omega}.$$
\end{enumerate}

\end{prop}

\begin{proof}
As before, set $u:=\kappa^2$. Then $u(t)$ oscillates between the values 
$(1\pm a)^2,$ satisfying equation \eqref{eq:kbounds},  
 $$
 \left({du\over dt}\right)^2 = (u+ b^2)( (1 + a)^2 - u )(u - (1 - a)^2).
 $$
Making the change of  variables $(t,u)\mapsto (x,y)$,  where 
 $$u = (1+a)^2 -4a y^2, \quad x=\omega t, \quad\omega={\sqrt{ (a+1)^2+b^2}\over 2}, 
 $$
one finds that $y(x)$ satisfies
$$
\left(\frac{dy}{dx}\right)^2 = (1-y^2)(1 - k^2y^2),
$$
where 
$$k^2={4a\over  (a+1)^2+b^2} .
$$
This is the ODE satisfied by the Jacobi elliptic function $y={\rm sn}(x,k)$,  with $y(0)=0$ and period $4K(k)$, where  $K(k)$ is the   complete elliptic integral of the first kind   with     {\em modulus}  $k$. Its square ${\rm sn}^2(x,k)$  has half that period, $2K(k)$. See  formula 22.13.1 and Table 22.4.2  of  \cite{RW}. \end{proof}

\begin{rmrk}
Equation  \eqref{eq:sn} of Proposition \ref{prop:exp} is  the same as in  \cite[page 614]{LS}, with $p^2=k^2$ given by equation \eqref{eq:params}  and 
$$w^2={(1+a)^2\over  (a+1)^2+b^2}.$$
\end{rmrk}

The torsion of the front tracks of the last proposition is given by equation \eqref{eq:geods2}. 

The curvature is periodic except for $a = 1, b = 0$, where one has $k = 1$ and $K(1) = \infty$. In this case $\x(t)$ is the Euler soliton and one has  
$$\kappa = 2{\rm cn}(t,1) = 2{\rm sech}(t).$$

We next give explicit formulas for the front tracks $\x(t)$ and their monodromy in cylindrical coordinates $r,\theta, z$ with respect to the rotation axis of $\K$.

\begin{prop}  \begin{enumerate}[{\rm (a)}]
\item
In  cylindrical coordinates $r,\theta, z$ with respect to the rotation axis of $\K$, the front track of a bicycling geodesic, with initial conditions $z(0) = 0$ and $\theta(0) = 0$,  is given by
\begin{align}
\label{eq:rcoords}
 r (t)&= {1\over  |\p|}\sqrt{A - 4a\,{\rm sn}^2(\omega t,k)},  \\
\label{eq:thetacoords}
\theta(t)& = {b\over 2 |\p|}\left[t + {B\over \omega}\, \Pi(\omega t, n, k)\right], \\
\label{eq:zcoords}
 z (t)&={1\over 2 |\p|}\left[ (|\p|^2 + 1) t - 4\omega\,E(\omega t,k)\right]
\end{align}
where
\[ 
 A =\frac{\left(|\p|^2+a\right)^2}{|\p|^2},\quad
 B = \frac{|\p|^2-a}{|\p|^2+a},\quad
 |\p| =\sqrt{ a^2 + b^2}, \quad n = \frac{4a}{A},
 \]
and $\omega, k$ are given  in equation \eqref{eq:params}. 

Here $E(x,k), \Pi(x,n,k)$ are the incomplete elliptic integrals of the second and third kind, given by 
\begin{align}
&E(x,k):=x-k^2\int_0^x{\rm sn}^2(s,k)ds, \label{eq:Exk}\\
& \Pi(x,n,k):=\int_0^x{ds\over 1-n\,{\rm sn}^2(s,k)}.\label{eq:Pxnk}
\end{align}

\item The monodromy is given by
\begin{align}
\label{eq:thetamon}
\Delta\theta& ={b\over \omega |\p|}  \left[ K(k) +  B\,\Pi(n, k)\right],\\
\label{eq:zmon}
\Delta z &= {1\over \omega |\p|}\left[(|\p|^2 + 1) K(k) - 4\omega^2 E(k)\right],
\end{align}
where $K(k), \ E(k):=E(K(k),k)$ and $\Pi(n,k):=\Pi(K(k),n,k)$ are the complete elliptic integrals of the first, second, and third kinds, respectively.
\end{enumerate}

\end{prop}

\begin{proof}
 From the proof of  Proposition \ref{prop:screw} we have  equation \eqref{eq:r}, 
\be \label{eq:req}
b^2+\kappa^2=|\p|^2r^2+{b^2\over |\p|^2},
\ee
from which one obtains equation \eqref{eq:rcoords} using equation \eqref{eq:sn} and $|\p| = \sqrt{a^2 + b^2}$.

Next, from equation \eqref{eq:pconst}, $2|\p| z'=1+a^2-\kappa^2,$ and upon substitution of equation \eqref{eq:sn}, we have
\be\label{eq:zp}|\p| z'=a\left[2{\rm sn}^2(\omega t,k)-1\right].
\ee
Equation \eqref{eq:zcoords}  follows by using formula \eqref{eq:Exk}.

Dotting both sides of \eqref{eq:rel} with $\x'$ and using the expression for $\K$ from \eqref{eq:mag_geod}, one obtains 
$$r^2|\p|\,\theta' = b\left( 1 - \frac{z'}{|\p|}\right),
$$
 from which, upon substitution of the above formulas \eqref{eq:zp} for $z'$ and \eqref{eq:req} for $ r$, we find
$$ |\p|\theta' = \frac{b}{2}\left[ 1 + \left(\frac{2(|\p|^2 + a)}{A} - 1\right)\frac{1}{1 - n \,{\rm sn}^2(\omega t,k)}\right],\mbox{ where } n := \frac{4a}{A}.
$$
Equation \eqref{eq:thetacoords} now follows from formula  \eqref{eq:Pxnk}.

Equations \eqref{eq:zmon}-\eqref{eq:thetamon}  now follow   by evaluating equations  \eqref{eq:thetacoords}-\eqref{eq:zcoords} at $t=T=2K/\omega$, and using   $E(2K(k),k)=2E(k),$
$\Pi(2K(k),n,k)=2\Pi(n,k).$
\end{proof}

\begin{rmrk}
 The explicit formulas \eqref{eq:rcoords}-\eqref{eq:zcoords}  are given for general Kirchoff rods in \cite{SH}, where they appear as equations (4.19a)-(4.19c).
\end{rmrk}

\input{references}
\end{document}

%% file: figures/bicycling1.pdf_tex
\begingroup%
  \makeatletter%
  \providecommand\color[2][]{%
    \errmessage{(Inkscape) Color is used for the text in Inkscape, but the package 'color.sty' is not loaded}%
    \renewcommand\color[2][]{}%
  }%
  \providecommand\transparent[1]{%
    \errmessage{(Inkscape) Transparency is used (non-zero) for the text in Inkscape, but the package 'transparent.sty' is not loaded}%
    \renewcommand\transparent[1]{}%
  }%
  \providecommand\rotatebox[2]{#2}%
  \newcommand*\fsize{\dimexpr\f@size pt\relax}%
  \newcommand*\lineheight[1]{\fontsize{\fsize}{#1\fsize}\selectfont}%
  \ifx\svgwidth\undefined%
    \setlength{\unitlength}{312.31203461bp}%
    \ifx\svgscale\undefined%
      \relax%
    \else%
      \setlength{\unitlength}{\unitlength * \real{\svgscale}}%
    \fi%
  \else%
    \setlength{\unitlength}{\svgwidth}%
  \fi%
  \global\let\svgwidth\undefined%
  \global\let\svgscale\undefined%
  \makeatother%
  \begin{picture}(1,0.2328409)%
    \lineheight{1}%
    \setlength\tabcolsep{0pt}%
    \put(0,0){\includegraphics[width=\unitlength,page=1]{bicycling1.pdf}}%
    \put(0.42427116,0.11273481){\color[rgb]{0,0,0}\makebox(0,0)[lt]{\lineheight{0}\smash{\begin{tabular}[t]{l}\s$\y(t)$\end{tabular}}}}%
    \put(0.69244535,0.08819999){\color[rgb]{0,0,0}\makebox(0,0)[lt]{\lineheight{0}\smash{\begin{tabular}[t]{l}\s$\x(t)$\end{tabular}}}}%
    \put(0,0){\includegraphics[width=\unitlength,page=2]{bicycling1.pdf}}%
    \put(0.5429612,0.13260149){\color[rgb]{0,0,0}\makebox(0,0)[lt]{\lineheight{0}\smash{\begin{tabular}[t]{l}\s$\v(t)$\end{tabular}}}}%
  \end{picture}%
\endgroup%

%% file: figures/vectors.pdf_tex
\begingroup%
  \makeatletter%
  \providecommand\color[2][]{%
    \errmessage{(Inkscape) Color is used for the text in Inkscape, but the package 'color.sty' is not loaded}%
    \renewcommand\color[2][]{}%
  }%
  \providecommand\transparent[1]{%
    \errmessage{(Inkscape) Transparency is used (non-zero) for the text in Inkscape, but the package 'transparent.sty' is not loaded}%
    \renewcommand\transparent[1]{}%
  }%
  \providecommand\rotatebox[2]{#2}%
  \ifx\svgwidth\undefined%
    \setlength{\unitlength}{125.31685791bp}%
    \ifx\svgscale\undefined%
      \relax%
    \else%
      \setlength{\unitlength}{\unitlength * \real{\svgscale}}%
    \fi%
  \else%
    \setlength{\unitlength}{\svgwidth}%
  \fi%
  \global\let\svgwidth\undefined%
  \global\let\svgscale\undefined%
  \makeatother%
  \begin{picture}(1,1.30233022)%
    \put(0,0){\includegraphics[width=\unitlength,page=1]{vectors.pdf}}%
    \put(0.5620388,0.84417436){\color[rgb]{0,0,0}\makebox(0,0)[lb]{\smash{\s $\x'_0$}}}%
    \put(0.2772636,0.84269764){\color[rgb]{0,0,0}\makebox(0,0)[lb]{\smash{\s$\p$}}}%
    \put(0.46042236,0.43122455){\color[rgb]{0,0,0}\makebox(0,0)[lb]{\smash{\s $\x_0$}}}%
    \put(0.23186988,0.2343637){\color[rgb]{0,0,0}\makebox(0,0)[lb]{\smash{\s $\v_0$}}}%
    \put(0.06115808,0.01379489){\color[rgb]{0,0,0}\makebox(0,0)[lb]{\smash{\s $\y_0$}}}%
    \put(0,0){\includegraphics[width=\unitlength,page=2]{vectors.pdf}}%
    \put(0.52117461,1.17984205){\color[rgb]{0,0,0}\makebox(0,0)[lb]{\smash{\s $\r_0$}}}%
    \put(0.67133597,1.18213409){\color[rgb]{0,0,0}\makebox(0,0)[lb]{\smash{}}}%
    \put(0,0){\includegraphics[width=\unitlength,page=3]{vectors.pdf}}%
  \end{picture}%
\endgroup%

%% file: figures/BC2-is.pdf_tex
\begingroup%
  \makeatletter%
  \providecommand\color[2][]{%
    \errmessage{(Inkscape) Color is used for the text in Inkscape, but the package 'color.sty' is not loaded}%
    \renewcommand\color[2][]{}%
  }%
  \providecommand\transparent[1]{%
    \errmessage{(Inkscape) Transparency is used (non-zero) for the text in Inkscape, but the package 'transparent.sty' is not loaded}%
    \renewcommand\transparent[1]{}%
  }%
  \providecommand\rotatebox[2]{#2}%
  \ifx\svgwidth\undefined%
    \setlength{\unitlength}{195.93905029bp}%
    \ifx\svgscale\undefined%
      \relax%
    \else%
      \setlength{\unitlength}{\unitlength * \real{\svgscale}}%
    \fi%
  \else%
    \setlength{\unitlength}{\svgwidth}%
  \fi%
  \global\let\svgwidth\undefined%
  \global\let\svgscale\undefined%
  \makeatother%
  \begin{picture}(1,0.68150309)%
    \put(0,0){\includegraphics[width=\unitlength,page=1]{BC2-is.pdf}}%
    \put(0.52665973,0.31840047){\color[rgb]{0,0,0}\makebox(0,0)[lb]{\smash{\s$\bf y$}}}%
    \put(0,0.14498436){\color[rgb]{0,0,0}\makebox(0,0)[lb]{\smash{\s$\bf x$}}}%
    \put(0.96942452,0.60480639){\color[rgb]{0,0,0}\makebox(0,0)[lb]{\smash{\s$\bf\tilde{x}=\x-2\v$}}}%
    \put(0.23384235,0.28997159){\color[rgb]{0,0,0}\makebox(0,0)[lb]{\smash{\s$\v$}}}%
    \put(0.65174674,0.55290881){\color[rgb]{0,0,0}\makebox(0,0)[lb]{\smash{\s$-\v$}}}%
  \end{picture}%
\endgroup%